\documentclass[12pt,longbibliography]{article}

\usepackage[utf8]{inputenc}

\usepackage[T1]{fontenc}
\usepackage[table,xcdraw]{xcolor}

\usepackage[a4paper, margin=2.7cm]{geometry}

\usepackage{amsmath, amsthm, amsfonts, amssymb}
\usepackage{mathrsfs}    

\usepackage[colorlinks=true,linkcolor=blue,citecolor=blue,urlcolor=blue,breaklinks]{hyperref}

\usepackage{bbm} 

\usepackage{graphicx,color}
\usepackage{float}
\usepackage{subcaption}

\usepackage{breakurl}
\usepackage{url}

\def\SM#1{\color{red}#1 \color{black}}

\providecommand{\keywords}[1]
{
	\small	
	\textbf{\textit{Keywords---}} #1
}

\def\R{\mathbb{R}}

\def\e{\varepsilon}

\DeclareMathOperator{\supp}{supp}

\def\p{\partial}

\def\:{\colon}

\newtheorem{thm}{Theorem}[section]

\newtheorem{lem}[thm]{Lemma}
\newtheorem{prp}[thm]{Proposition}

\theoremstyle{definition}
\newtheorem{dfn}[thm]{Definition}
\theoremstyle{remark}

\theoremstyle{example}

\def\thetitle{Impact of the horizontal gene transfer on the evolutionary equilibria of a population}
\def\theauthor{Alejandro Gárriz, Alexis Léculier, Sepideh Mirrahimi}

\hypersetup{pdftitle={\thetitle}}
\hypersetup{pdfauthor={\theauthor}}

\title{\thetitle}
\author{\theauthor}

\AtEndDocument{\bigskip{\footnotesize
		\noindent
		\textsc{Alejandro Gárriz. IMAG Institut Montpelliérain Alexander Grothendieck, 34090 Montpellier, France.}
		\textit{E-mail address}: \texttt{\href{mailto:alejandro.garriz@math.uni-toulouse.fr}{alejandro.garriz@math.univ-toulouse.fr}}
		\par
		\addvspace{\medskipamount}
		\noindent
		\textsc{Alexis Léculier. Université de Bordeaux, 33000 Bordeaux, France.}
		\textit{E-mail address}: \texttt{\href{mailto:alexis.leculier@u-bordeaux.fr}{alexis.leculier@u-bordeaux.fr}}
		\par
		\addvspace{\medskipamount}
		\noindent
		\textsc{Sepideh Mirrahimi. IMAG Institut Montpelliérain Alexander Grothendieck, 34090 Montpellier, France.}
		\textit{E-mail address}: \texttt{\href{mailto:sepideh.mirrahimi@umontpellier.fr}{sepideh.mirrahimi@umontpellier.fr}}
}}

\begin{document}
	
	\maketitle
	
	 \begin{abstract}
	 How does the interplay between selection, mutation and horizontal gene transfer modify the phenotypic distribution of a bacterial or cell population? While  horizontal gene transfer,  which corresponds to the exchange of genetic material between individuals, has a major role in the adaptation of many organisms, its  impact on the phenotypic density of  populations is not yet fully understood.  
	 
	 We study an elliptic integro-differential equation describing the evolutionary equilibrium of the phenotypic density of an asexual population.   In a regime of small mutational variance, we characterize the solution which results from the balance between   competition for a resource, mutation and horizontal gene transfer.  We show that in a certain range of  parameters polymorphic equilibria exist, which means that the phenotypic density may concentrate around several dominant traits. Such polymorphic equilibria result from an antagonist interplay between  horizontal gene transfer and selection, while similar models which neglect  the transfer lead only to monomorphic equilibria.

	\end{abstract}
	
	
	\tableofcontents
	\section{Introduction}

	\subsection{Model and biological motivations}
	
Horizontal gene transfer, which corresponds to the exchange of genetic materials between individuals, has a major role in the evolution and adaptation of many organisms, as for instance in the evolution of bacterial virulence or antibiotic resistance \cite{HO.JL.EG:00}.  An important example of horizontal transfer mechanism concerns bacterial plasmids which can modify significantly the fitness of their hosts \cite{DR.ER.SB:11,FS.BL:77}. Plasmids are small circular doubly stranded DNA, physically separated from the chromosomal DNA, which may be replicated and transferred  from a cell to another, when they are in contact, independently of the chromosome. They carry factors that can be beneficial for the survival of the bacteria and lead to a selective advantage, as for genes for antibiotic resistance. However, they also have fitness costs, like reduced reproduction rate. What is the  outcome of the trade-off between the fitness costs of the plasmids and their advantage by their accelerated spread?\\

Recent studies have shown that the interplay between mutation, horizontal gene transfer and selection may lead to new behaviors when compared to models considering only selection and mutation, where the population is usually driven to the fittest traits \cite{SB.PC.RF:16,BilliardFerriereMeleardTran,CalvezFigueroaAl}.  Such an interplay may for instance lead to extinction of a population or re-emergence of an apparently extinct trait and a cyclic behavior of the population  \cite{BilliardFerriereMeleardTran}. In such a situation, horizontal transfer, if it occurs with a rather strong rate, drives the population to unfit traits. Then, whether a small apparently extinct subpopulation with a fitter trait re-emerges and repopulates the environment (this is called an evolutionary rescue), or the population goes extinct (this is called an evolutionary suicide).  Note that the first scenario may for instance be interpreted as a re-emergence of antibiotic resistance, while the second one may correspond to a successful treatment.  While such types of behaviors are observed numerically in \cite{BilliardFerriereMeleardTran,CalvezFigueroaAl}, a   theoretical understanding  of them  is still lacking (see however  \cite{SB.PC.RF:16,NC.SM.CT:21} where some stochastic models considering   a finite number of strains have been studied).

We consider the following integro-differential model describing the dynamics of the phenotypic density of 	an asexual population subject to mutation, selection and horizontal transfer:
	\begin{equation}\label{eq:time_dependant_main}
		\begin{cases}
			\partial_t  n( t,  z) = \sigma\ \partial^2_{zz} n (t,  z)+\left(R(z) -  \kappa \rho( t)\right) n( t, z) +  \tau\cdot n( t, z)\displaystyle\int_\R \frac{ n( t,  y)}{ \rho( t)}\cdot H(K( z - y))\ {\rm dy}\\
			n( 0,  z) = n_0(z),\\
			n(t, z)>0,\\
			\rho(t)=\displaystyle\int_\R n(t,y)\ {\rm dy}.
		\end{cases}
	\end{equation}
	A variant of this model was derived from stochastic individual based model in \cite{BilliardFerriereMeleardTran}.
	Here, $n(t,z)$ stands for the phenotypic density of a population, with   $t\in \R^+$ and $z\in \R$ corresponding respectively to time and a phenotypic trait. The diffusion term models the mutations which generate phenotypic variability in the population. Individuals grow at rate $R(z)$ and are regulated by a uniform competition for resources with intensity $\kappa$. The last nonlinear and nonlocal term in the right hand side of the equation corresponds to the horizontal transfer term. More precisely, $\tau$ denotes the transfer rate and   $H(K(z-y))$ denotes the transfer flux from trait $y$ to trait $z$, with  $K$  a steepness parameter for the transfer flux (see Section \ref{sec:assumptions} to understand its role).\\
 
	In this work we focus on the qualitative properties of the stationary solution of the equation above, that is
	\begin{equation}
     \label{steady}
	-\sigma n'' ( z)=\left(R(z) -  \kappa\rho\right) n(z) +  \tau n(z)\int_\R \frac{ n(  y)}{ \rho}H(K( z - y))\ {\rm dy}.	    
	\end{equation}
We provide an asymptotic analysis of the equation above considering small mutational effects, that is $\sigma<<1$. 
The study of the steady solution is a first step in the theoretical description of the behaviors observed numerically in \cite{BilliardFerriereMeleardTran,CalvezFigueroaAl}, which studied stochastic and deterministic models close to \eqref{eq:time_dependant_main}. We will see below that our work allows to identify a new feature, that is the existence of polymorphic steady solutions, which was not observed in the previous numerical results \cite{BilliardFerriereMeleardTran,CalvezFigueroaAl} .

 \subsection{State of the art}
 
 Models of horizontal gene transfer have been studied considering a finite number of strains and using ordinary differential equations \cite{LEVIN1979, FS.BL:77}, or in a population genetics context without ecological concern \cite{AN.GK.EK:05,ST.SB:13}. In \cite{PH.LF.PM.GW:09,PM.GR:15} some integro-differential models of horizontal transfer have been studied  in a different context than our work.  Our study follows a series of works based on stochastic individual based models motivated by the eco-evolutionary dynamics of bacterial plasmids \cite{SB.PC.RF:16,BilliardFerriereMeleardTran,CalvezFigueroaAl,NC.SM.CT:21}. More specifically, in \cite{BilliardFerriereMeleardTran} a stochastic individual based model was introduced considering a quantitative trait. It was shown that in the limit of large populations such an individual based model converges to an equation close to \eqref{eq:time_dependant_main}, where the mutation term is modeled via an integral kernel rather than a Laplace term.  Three types of behavior where identified in the numerical simulations in \cite{BilliardFerriereMeleardTran}. In the first scenario, the phenotypic density would concentrate around a trait close to the trait with maximal growth rate. In the second scenario,  we observe a cyclic behavior. The population concentrates around an evolving trait. The horizontal transfer   drives the dominant trait to an unfit value. Then, a small apparently extinct subpopulation with a fitter trait emerges. This emergent trait is again driven to an unfit trait and such dynamics are observed repeatedly. In the third scenario, the population concentrates first around an evolving trait. The horizontal transfer drives the dominant trait to an unfit value, this time no small subpopulation emerges and the population goes extinct. In \cite{CalvezFigueroaAl}, theses stochastic simulations where compared with the numerical resolution of an integro-differential model derived in \cite{BilliardFerriereMeleardTran}, considering small mutational effects and similar types of behavior where observed. In \cite{NC.SM.CT:21} a stochastic model with a finite number of strains was studied theoretically. In a particular case of three strains, a periodic behavior was captured in a certain range of parameters.\\
 
In this paper, we provide an asymptotic analysis of  the steady solution of \eqref{eq:time_dependant_main}, that is the solution to \eqref{steady}, considering small mutational effects. This is a first step to provide a theoretical description of the behaviors observed numerically in \cite{BilliardFerriereMeleardTran,CalvezFigueroaAl}. The choice of the Laplace term, instead of an integral kernel for the mutation term has been done to reduce the technicality of the analysis. We believe that this choice would not modify the qualitative behavior of the solution in the limit of vanishing mutations. \\

To perform our analysis we use an approach based on Hamilton-Jacobi equations. This approach was first introduced in \cite{DJMP} and then widely developed to study models of quantitative traits from evolutionary biology (see for instance \cite{BarlesMirrahimiPerthame, BarlesPerthame1, BarlesPerthame2}). A closely related approach was also previously used in the geometric optics approximation of solutions of reaction-diffusion equations (see for instance \cite{ EvansSouganidis, Freidlin}). Here, we extend this approach to the study of horizontal gene transfer. Note that some heuristic computations using this approach were provided in \cite{CalvezFigueroaAl}.

	\subsection{Assumptions}
	\label{sec:assumptions}
	
	We make the following assumptions on the transfer term $H$:
	\begin{equation}\label{eq:hypothesis_H}\tag{H1}
	\begin{aligned}
		(1)\quad &H \in C^3(\mathbb{R}) \text{ is odd and monotone increasing from -1 to 1}.\\
		(2)\quad &H(0)=0, H'(0)=1, H''(z)<0\text{ for all }z>0 .\\
		(3)\quad &\text{There exists a positive }z_H\text{ such that for all }|z|>z_H, H'''(z)>0,\\
		&\text{while for all }|z|\leq z_H, H'''(z)\leq 0.
	\end{aligned}
	\end{equation}
	 The examples that we have in mind are the functions 
	$$
	H(z)=\tanh(z)\qquad \text{or}\qquad H(z)=\frac{2}{\pi}\arctan(z).
	$$
	One can think of this kernel as $H(z-y) = \alpha(z-y) - \alpha(y-z)$ with $\alpha$ a smooth function that behaves like a Heaviside step function. {Then, one would consider that the transfer arises only from larger traits $y$ to smaller traits $z$, with $z<y$ and the transfer rate between $y$ to $z$ would be given by $\alpha(z-y)$. This choice of transfer term is motivated by the example of plasmids which are transmitted from one bacterium to another by cell-to-cell contact. }\\
 
	Next, the values $\tau$ and $K$ are considered to be strictly positive, i.e.,
	$$
	\tau>0\quad\text{and}\quad K>0.
	$$
	The value $\tau$ is understood as the strength of the transfer, while the value $K$ in~\eqref{eq:time_dependant_main} corresponds to the steepness of the  transfer rate. Note that as $K\to\infty$ the transfer rate $H(K z)$ approaches the Heaviside step function.\\
 
	Lastly, we provide our assumptions on the growth term $R(z)$:
	\begin{equation}\tag{H2}
	\label{eq:hypothesis_R}
		\begin{aligned}
			(1)\quad &R\in C^2(\R),\\
			(2)\quad  &\text{There exists a bounded domain }D_R \text{ such that }R(z)\geq 0\text{ for all }z\in \overline{D_R}\\
			&\text{ and }R(z)< 0\text{ for all }z\in D_R^c,\quad\text{and}\\
			(3)\quad &\lim\limits_{|z|\to\infty}R(z)=-\infty.
		\end{aligned}
	\end{equation}
	A typical example is given by
\begin{equation}
\label{R:quadratic}
  R(z)=1-gz^2,  
\end{equation}
which will be studied in detail later on in the article.
	
	\subsection{Preliminary tools}
	
	\subsubsection{An adimensional parameterization of the problem }\label{subsec:change_of_variables}

	We introduce a dimensionless parameterization of the problem via the following  change of variables
	$$
	\tilde{z}= K z,\quad \tilde n( \tilde z)=\frac{\kappa}{r K}\cdot n\left(\frac{\tilde z}{K}\right), \quad \varepsilon^2=\frac{\sigma K^2}{r}, \quad \tilde R(\tilde z)=\frac{R(\frac{\tilde z}{K})}{r} ,\quad\text{and}\quad \tilde \tau = \frac{\tau}{r},
	$$
	where $r$ is defined as
	$$
	r:=\max\limits_{z\in \R} R(z).
	$$
    The problem \eqref{steady} is then written (we drop the tildes   for the sake of readability) as
	\begin{equation}\label{eq:main}
		\begin{cases}
			-\varepsilon^2 n''_\varepsilon (z)=\left(R(z) -  \rho_\varepsilon \right) n_\varepsilon(z) +  \tau\cdot n_\varepsilon( z)\displaystyle\int_\R \frac{ n_\varepsilon( y)}{ \rho_\varepsilon}H( z - y)\ {\rm dy}\\
			n_\varepsilon( z)>0,\\
			\rho_\varepsilon=\displaystyle\int_\R n_\varepsilon(y)\ {\rm dy}.
		\end{cases}
	\end{equation}
    Note that in this new version, the selection term is re-normalized such   that
	$$
	\max_{z\in\R}R(z)=1.
	$$
	Note also that if we were considering the time-dependent equation from~\eqref{eq:time_dependant_main} we would make the change of variables
	$$
	\tilde t= rt.
	$$	
	In the particular case where $R(z)=r-gz^2$ we also consider the following change of variable
	$$
	\tilde g=\frac{g}{r K^2}
	$$
	which leads, again after dropping the tilde for the sake of readability, to
	\begin{equation}\label{eq:main_growth}
		\begin{cases}
			-\varepsilon^2  n''_\varepsilon (z)=\left(1-gz^2 -  \rho_\varepsilon \right) n_\varepsilon(z) +  \tau\cdot n_\varepsilon( z)\displaystyle\int_\R \frac{ n_\varepsilon( y)}{ \rho_\varepsilon}H( z - y)\ {\rm dy}\\
			n_\varepsilon( z)>0,\\
			\rho_\varepsilon=\displaystyle\int_\R n_\varepsilon(y)\ {\rm dy}.
		\end{cases}
	\end{equation}
	
	\subsubsection{An eigenvalue problem}
	
	Before presenting our   main results  let us introduce the following eigenvalue problem 
	\begin{equation}\label{eq:eigenpair_whole_space}
		\begin{cases}
			-\varepsilon^2 N''_\varepsilon(z) - R(z)N_\varepsilon(z) = -\lambda_\varepsilon(\R) N_\varepsilon(z),\quad z\in\Omega\\
			N_\varepsilon(z)>0, \quad z\in \Omega, \quad \|N_\varepsilon\|_{L^2(\R)}=1.	
		\end{cases}
	\end{equation}
   with $\lambda_\varepsilon(\R)\in \mathbb{R}$   the principal eigenvalue of the problem and 
	 $N_\varepsilon\in H^1(\R)$ the principal eigenfunction. By classic  theory, since $R(z)$ is a confining term, there exists a unique eigenpair $(N_\varepsilon,\lambda_\varepsilon(\R))$ and in fact
	\begin{equation}\label{eq:eigenvalue}
		-\lambda_\varepsilon(\R) = \min\limits_{u\in H^1(\R), u\neq 0}\frac{1}{\|u\|^2_{L^2(\R)}}\left\{ \int_\R\varepsilon^2 ( u')^2- Ru^2 {\rm dz}\right\}.
	\end{equation}
	Finally, we have the following Lemma, which is proven in Section~\ref{sec:dirichlet}.
 \begin{lem}
 \label{lem:limit_eigenvalue}
     As $\varepsilon\to 0$, $\lambda_\varepsilon(\R)$ converges to $ 1$. Moreover, we have $\lambda_\varepsilon(\R)\leq 1$.
 \end{lem}

	\subsection{Main results}
	
	\subsubsection{Existence of solutions}
	Let us present the main results of our study, commencing with the existence theory. Here, we provide a sufficient condition to obtain existence of a non-trivial solution to \eqref{eq:main}. To this end, we first introduce the following assumption
 \begin{equation}
 \label{H:tau}
      \tau<\lambda_\varepsilon(\R).
 \end{equation}
	\begin{thm}\label{thm:existence}
		Let $\varepsilon>0$ and assume \eqref{eq:hypothesis_H}, \eqref{eq:hypothesis_R} and \eqref{H:tau}. Then there exists a non-trivial solution $n_\varepsilon$ to problem~\eqref{eq:main}. Moreover this solution satisfies
		$$
		\max\limits_{z\in\R} n_\varepsilon(z)\geq  \tilde\delta(1-\tau),
		$$
		where $\tilde\delta$ is a positive constant depending only on the function $R$ and the constant $\tau$,
		\begin{equation}
		\label{boundrho}
		0<\lambda_\varepsilon(\R)-\tau \leq \rho_\varepsilon\leq 1 + \tau,
		\end{equation}
		and
		\begin{equation}\label{eq:mass_varepsilon}
			\displaystyle \int_\R R(z) n_\varepsilon(z)\  {\rm dz}-\rho_\varepsilon^2=0.
		\end{equation}
	\end{thm}
Note   that the eigenpair problem \eqref{eq:eigenpair_whole_space} is equivalent with
  problem~\eqref{eq:main} when   no horizontal transfer is considered (that is $\tau=0$). One can indeed construct a solution to \eqref{eq:main} in the case $\tau=0$ from the eigenpair $(\lambda_\varepsilon,N_\varepsilon)$  via the formula $n_\varepsilon(z)=\lambda_\varepsilon(R)\frac{N_\varepsilon(z)}{\int N_\varepsilon(z)dz}$, so that $\lambda_\varepsilon(\R)$ corresponds to the total population size  $\rho_\varepsilon$. We conclude that there exists a non-trivial positive solution to the problem without transfer if and only if $\lambda_\varepsilon(\R) > 0$. We also notice that 
	$$
	\left|\tau\displaystyle\int_\R \frac{n(y)}{\rho}H(z-y)\ {\rm dy}\right|\leq \tau.
	$$
 These properties allow us to prove the existence of a non-trivial solution under assumption \eqref{H:tau}.
 However,  we do not expect this condition to be sharp for the existence to hold, since numerical simulations suggest otherwise. 
 
 \subsubsection{Asymptotic behavior of the solution}
	We next study the asymptotic behavior of the solutions as the mutational effect $\varepsilon$ vanishes.  We expect the solution to concentrate, as the diffusion term vanishes, around certain dominant traits, forming Dirac's delta functions in the limit. In order to identify such singular limits, we use an approach based on Hamilton-Jacobi equations \cite{BarlesMirrahimiPerthame, BarlesPerthame2,DJMP}. The main ingredient in this approach is to perform a Hopf-Cole transformation:
 	\begin{equation}\label{eq:transformed_hopf_cole}
	u_\varepsilon(z):=\varepsilon\cdot \ln\big(n_\varepsilon(z)\big),
\end{equation}
 which allows to unfold the singularity of the problem. Indeed, while $n_\varepsilon$ tends, as $\varepsilon\to 0$, to a singular measure, $u_\varepsilon$ converges to a continuous function $u$ which solves a Hamilton-Jacobi equation. The main idea is then to first study the limit of $u_\varepsilon$ and next to use some information on the function $u$ to identify $n$. 
 We prove the following.

	\begin{thm}\label{thm:limit_epsilon}
		Let $\tau< 1$. As $\varepsilon\to 0$ and along subsequences,   $\rho_\varepsilon$ converges to a positive function $\rho_0$ and  $u_{\varepsilon }$ converges locally uniformly   to a continuous function $u$ that is semi-convex and a viscosity solution of
		\begin{equation}\label{eq:viscosity_u}
			\begin{cases}
				- ( u'(z))^2 = R(z) - \rho_0 +\Phi_0(z),\quad z\in\R,\\
				\max\limits_{z\in\R}u(z)=0,\\
			\end{cases}
		\end{equation}
		where $\Phi_0\in C^3:\R\to (-\tau,\tau)$ and
		$$
		 1-\tau\leq \rho_0\leq 1+\tau.
		$$
		Moreover, as $\varepsilon\to 0$ and along subsequences,  $n_\varepsilon$ converges to a measure $n$. When the limits are considered along the same subsequences, we have the following relations between $u$, $\Phi_0$ and $n$:
		  \begin{equation}
		\label{phi0}
	     \rho_0= \int_\R n(z)dz>0,\quad \Phi_0 (z)=\int \frac{n(y)}{\rho_0}H(z-y)dy,
		  \end{equation}
        \begin{equation}
        \label{inclusion}
        \supp n(z)\subseteq \{z\in\R : u(z)=0\}\subseteq \left\{z\in\R : R(z)-\rho_0+\Phi_0(z)=0\right\}.
        \end{equation}

	\end{thm}

\subsubsection{The limit profile}

	
	Up to this point we developed the theory for a general confining term $R$. For the rest of the results we focus on the particular case
	$$
	R(z)=1-gz^2,
	$$
	which is derived from the more general term $R(z)=r-gz^2$ via the reparametrization mentioned earlier.

    In order to provide our qualitative results on the limit $n$, we need to introduce some definitions. We first introduce the following function 
    \[ F(z)  = 1 - gz^2 - \rho_0 + \tau\int_{\mathbb{R}} \frac{n(y)}{\rho_0} H(z-y)dy \qquad \text{ with } \rho_0 = \int_{\mathbb{R}} n(y)dy\]
    which corresponds to the r.h.s. of \eqref{eq:viscosity_u}. We will refer to $F(z)$ as the fitness function. We next define the notion of \textit{Evolutionary Stable Strategy}. 
    \begin{dfn}
    \label{definition:ESS}
        A phenotypic density $n$ corresponds to an \textit{Evolutionary Stable Strategy}, or \textit{ESS} to abbreviate, if the following conditions are satisfied.
	\begin{equation}\label{def:ESS}
		\left\lbrace 
		\begin{aligned}
			F(z) &\leq  0 && \text{ for all } z \in \mathbb{R}\setminus \mathrm{supp} \ n, \\
			F(z) & = 0 && \text{ for all } z \in \mathrm{supp} \ n.
		\end{aligned}\right.
	\end{equation}
  The support of $n$ is then called  the  \textit{Evolutionary Stable Strategy}. If this support is discrete we will talk about ESS points.
  Moreover, we will say that this ESS is $m-$morphic if 
	\[ \#\lbrace \mathrm{supp} \ n \rbrace = m. \]
	If $m=1$ we will speak of monomorphism, if $m=2$ of dimorphism and so on.
    \end{dfn}
    From \eqref{eq:viscosity_u} and \eqref{inclusion} we deduce the following. 
    \begin{prp}
    \label{prop-ESS}
	    Let $n_\varepsilon$ be a solution of~\eqref{eq:main} that converges, as $\e\to 0$ and along a subsequence, to a measure $n$. Then $n$ corresponds to an ESS. 
	\end{prp}
	Briefly, this is because to each $n$ obtained as a limit of $n_\e$ corresponds, by the Hopf-Cole transformation, one $u$ obtained as the limit of the $u_\e$ that satisfies~\eqref{eq:viscosity_u} and~\eqref{inclusion}. The fitness function $F(z)$ is precisely equal to $-|u'(z)|^2$.
 
   Note that the notion of \textit{Evolutionary Stable Strategy} is taken from the field of adaptive dynamics \cite{OD:04, SG.EK.GM.JM:98}, which focuses on a different time scale where the mutations are very rare so that they arise one by one and between two mutations the population attains its equilibrium. Here, we do not consider such a framework. However, when considering the steady solutions and vanishing mutational effects, we recover the evolutionary stable strategies of adaptive dynamics (see for instance \cite{SM:17,SG.SM:18} where such a property has been obtained in other contexts). 
   
	Finally, before showing the last main theorem, we define the quantity
	\begin{equation}\label{eq:mu}
	\mu:=\frac{\tau}{2g}.
	\end{equation}
	This ratio measures the interplay between the strength of selection and the nonlocal horizontal transfer, and it appears naturally in the identification of the evolutionary stable strategies.
	
	\begin{thm}\label{thm:main}
		  There exist positive constants $\mu_1$ and $\mu_2$ such that the following results hold.
		\begin{enumerate}
			\item There exists a unique monomorphic ESS if and only if $0 \leq \mu \leq \mu_1$ and $\tau<\frac{2}{\mu}$. 
			\item There exists a unique dimorphic ESS if and only if $\mu_1 < \mu \leq \mu_2$ and $\tau<\tau_2$, where $\tau_2$ is a positive value depending on $H$ and $\mu$.
		\end{enumerate}
		Moreover, the \textit{ESS} points, the values $\mu_1$ and $\mu_2$ and the associated demographic equilibria are fully characterized by $H$ and the parameters $\tau$ and $g$. 
	\end{thm}
	We have to precise that the last affirmation regarding the dimorphic case is proven assuming one extra hypothesis on $H$ that will be presented in due time, in Section \ref{sec:dimo}, and that, while hard to write down for a general transfer kernel, is quite easy to verify for a particular choice of $H$. 
	
The theorem above provides the range of parameters for which monomorphic and dimorphic ESS exist. Note however that the theorem does not guarantee that with this range of parameters these are the only possible evolutionary stable strategies. One could wonder whether the   solution of \eqref{steady} is indeed close to such monomorphic or dimorphic ESS. To test this hypothesis, we solved numerically a time dependent version of ~\eqref{eq:main},  that is 
\begin{equation}
\label{eq:time} 
	\begin{cases}
		\varepsilon\p_t n_\e(t,z)	-\varepsilon^2 \p^2_{zz}n_\varepsilon (t,z)=\left(R(z) -  \rho_\varepsilon(t) \right) n_\varepsilon(t,z) +  \tau\cdot n_\varepsilon( z)\displaystyle\int_\R \frac{ n_\varepsilon( t,y)}{ \rho_\varepsilon}H( z - y)\ {\rm dy}\\
			n_\varepsilon( 0,z)=n_{\varepsilon,0}(z),\\
			\rho_\varepsilon(t)=\displaystyle\int_\R n_\varepsilon(t,y)\ {\rm dy},
	\end{cases}
\end{equation}
and studied the long time solution considering the following particular form of transfer function 
$$
H(z)=\tanh(z).
$$
Note that the $\e$ in front of $\p_t n_\e$ in \eqref{eq:time} comes from a change of variable in time $t\to \frac t\e$. This is a classical change of variable in such type of models \cite{BarlesMirrahimiPerthame, BarlesPerthame1, BarlesPerthame2}. Indeed since the mutational effects are supposed to be small, the evolutionary dynamics are expected to be slow. This change of variable allows to capture the  dynamics  of the phenotypic density taking into account the small effects of the mutations.

In Figure~\ref{fig:mono}, we illustrate the numerical solution of  \eqref{eq:time}  with a choice of parameters such that $\mu<\mu_1$. As we can see in Figure~\ref{fig:mono},  the solution concentrates around a trait that travels to higher values as times goes by and eventually converges to a concentrated distribution around  $z_0$. The total population size converges quite smoothly to the theoretical expected value and the solution $n_\varepsilon$ at the final time resembles a Dirac's delta.  The numerical scheme that we used is an adaptation of the scheme developed in \cite{CalvezHivertYoldacs, CalvezFigueroaAl} and we expect it to be asymptotic preserving and hence adapted to deal with small values of $\e$. See Appendix \ref{section:annex} for details. All the pictures present in this article come from numerical simulation done in the interval $[-2,6]$ for the trait variable, with step sizes $\Delta t=10^{-4}$ and $ \Delta z=10^{-2}$.
	\begin{figure}[H]\centering
        \renewcommand\thesubfigure{\Alph{subfigure}}
		\subfloat[Function $n_\varepsilon(t,z)$]{ \includegraphics[scale = 0.55]{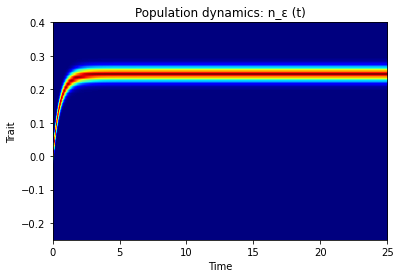}}
        \subfloat[Function $\rho(t)$]{ \includegraphics[scale = 0.55]{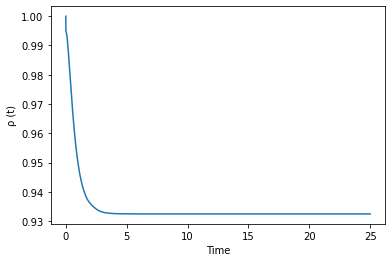}}\\[10pt]
		\subfloat[ $u_\varepsilon$ and $n_\varepsilon$ (rescaled) at the final time]{ \includegraphics[scale = 0.55]{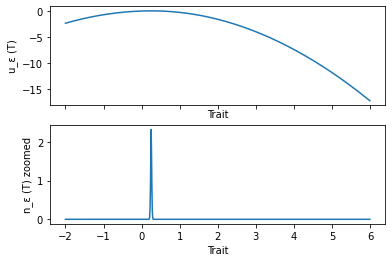}}
		\caption{Solution of~\eqref{eq:time} in the monomorphic case $\mu=0.25$, with values $\tau=0.5$, $g=1$ and $\varepsilon=5\cdot 10^{-5}$. In picture (A) the colors correspond to the isolines of the phenotypic distribution. In picture (C) the function $n_\varepsilon$ is rescaled in order to better appreciate other possible maxima.}
        \label{fig:mono}
	\end{figure}
		In Figure~\ref{fig:di} we illustrate a second example with $\mu_1<\mu<\mu_2$ such that we expect dimorphism. We can observe that 
		the solution concentrates first around a dominant trait that is driven by horizontal transfer to larger values. At a certain time, the dominant trait becomes too unfit (with a small growth rate $R$) such that 
		the population size drops. Then, some fitter traits emerge and start traveling again to larger values. The interplay between horizontal transfer and selection produces several  jumps back and forth of the solution up until one point when it stabilises by reaching two maxima, as appreciated in the picture of $n_\varepsilon$ at the final time. As we can also see, the mass $\rho(t)$ does not converge so smoothly but oscillates quite wildly (as a consequence of this process of stepwise evolution) before finally converging.
	\begin{figure}[H]\centering
        \renewcommand\thesubfigure{\Alph{subfigure}}
		\subfloat[\label{A}][Function $n_\varepsilon(t,z)$]{ \includegraphics[scale = 0.55]{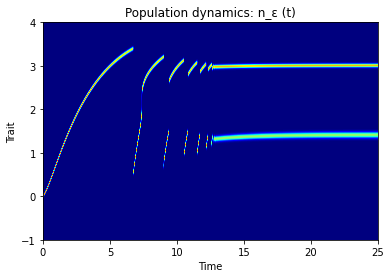}}
        \subfloat[\label{B}][Function $\rho(t)$]{ \includegraphics[scale = 0.55]{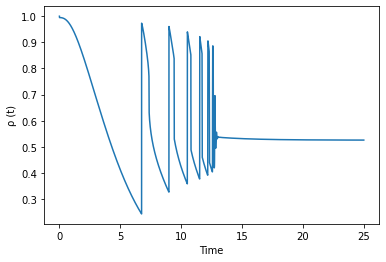}}\\[10pt]
		\subfloat[\label{C}][$u_\varepsilon$ and $n_\varepsilon$ (rescaled) at the final time]{ \includegraphics[scale = 0.55]{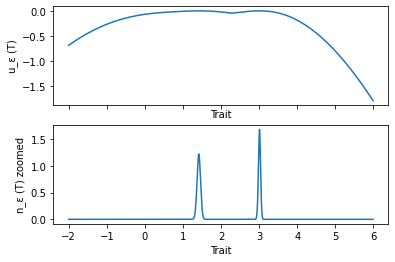}}
		\caption{Solution of~\eqref{eq:time} in the dimorphic case $\mu=3.84$, with values $\tau=0.5$, $g=0.065$ and $\varepsilon=5\cdot 10^{-5}$.}
        \label{fig:di}
	\end{figure}

%
	
	One could wonder what would be the shape of the long-time solution when $\mu>\mu_2$. We believe that polymorphism goes on as $\mu$ increases, giving rise to the following conjecture.
	
	\noindent\textbf{Conjecture 1:} \textit{ There exists a strictly increasing sequence of values  $\mu_m \geq 0$ such that there is an $m-$morphic ESS if and only if } \[\mu_{m-1} < \mu \leq \mu_m\quad\text{and}\quad \tau<\tau_\mu \]
	where $\tau_\mu$ is a positive value depending only on $H$ and $\mu$.
	
	To illustrate this point we have dedicated some time to study the trimorphic case from a theoretical point of view, but while the method we present scales nicely with the cardinality of $\supp n$, the analysis becomes at this point too complex to be solved theoretically. We had to resort to numerical simulations in an effort to characterize at least the trimorphic case for the particular case where $H(z)=\tanh(z)$ (see Section \ref{sec:trimo}). In Figure~\ref{fig:tri} we illustrate a situation with $\mu=5$ where, thanks to  the numerical study in Section \ref{sec:trimo}, we expect  trimorphism. We observe in Figure~\ref{fig:tri} that the numerical solution of \eqref{eq:main} indeed converges to the expected trimorphic ESS, matching the values for $\rho$ and the points of the ESS calculated for the case $\mu=5$ in Table~\ref{table-tri}. 
	
	\begin{figure}[H]\centering
        \renewcommand\thesubfigure{\Alph{subfigure}}
		\subfloat[\label{A}][Function $n_\varepsilon(t,z)$]{ \includegraphics[scale = 0.55]{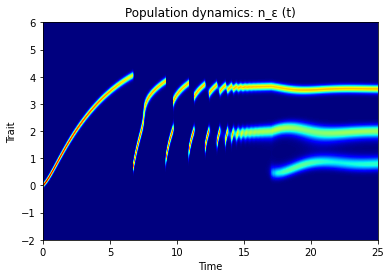}}
        \subfloat[\label{B}][Function $\rho(t)$]{ \includegraphics[scale = 0.55]{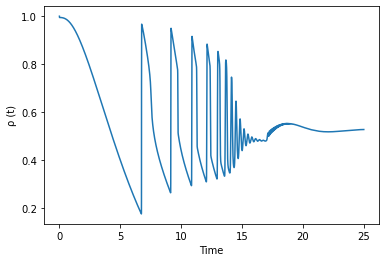}}\\[10pt]
		\subfloat[\label{C}][ $u_\varepsilon$ and $n_\varepsilon$ (rescaled) at the final time]{ \includegraphics[scale = 0.55]{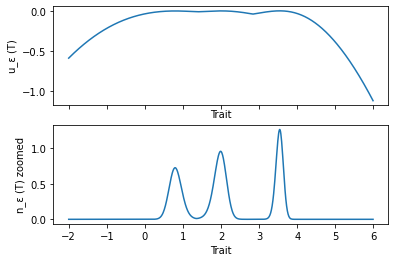}}
		\caption{Solution of~\eqref{eq:time} in the trimorphic case $\mu=5$, with values $\tau=0.5$, $g=0.05$ and $\varepsilon=5\cdot 10^{-4}$.}
        \label{fig:tri}
	\end{figure}
	
	\subsection{Comments and comparison with the previous results}

 Our results differ in two ways with previous numerical results in \cite{BilliardFerriereMeleardTran,CalvezFigueroaAl}. First, our work illustrates theoretically and numerically that polymorphic phenotypic distributions emerge as a result of the trade-off between selection and horizontal transfer. Such polymorphic distributions were not observed in \cite{BilliardFerriereMeleardTran,CalvezFigueroaAl}. Second, the numerical results in \cite{BilliardFerriereMeleardTran,CalvezFigueroaAl} suggested that when the dominant trait does not converge to a fixed value, the phenotypic density would have a cyclic behavior. The phenotypic density can indeed be driven to unfit traits due to horizontal transfer, and then have a jump to a fitter trait that evolves again to unfit traits and this behavior occurs repeatedly, suggesting a periodic behavior of the solution in long time. While we also observe in our numerical results such a stepwise evolution, this behavior seems transitory before the convergence of the solution to a polymorphic state (see Figures 2 and 3). Note however that the numerical simulations in \cite{BilliardFerriereMeleardTran,CalvezFigueroaAl} where run for a short time. Hence it is possible that in those simulations also the cyclic behavior would be transitory.

Let us provide some details that would help to understand the differences between these works. In \cite{CalvezFigueroaAl} an equation  closely related  to \eqref{eq:time_dependant_main} was studied numerically, considering small mutational effects. In that article, an integral term was used to model the mutations instead of the diffusion term in the present work. We believe that this choice of mutation term would not effect the asymptotic shape of the solutions since similar analysis may be done, at least formally, in the case of a mutation kernel. A more important difference between these works is that the authors of \cite{CalvezFigueroaAl} consider a rather large steepness parameter $K$  in order to approximate a Heaviside transfer kernel, and they do not rescale the problem as we did in Section \ref{subsec:change_of_variables}.  If we undo the change of variables of Section~\ref{subsec:change_of_variables} we find that
	$$
	\mu=K^2 \frac{\tau}{2g}.
	$$
	Increasing $K$ in order to approximate a Heaviside-like transfer kernel would result in a very big $\mu$ while at the same time we are rescaling the trait space by $z=\frac{\tilde z}{K}$, bringing the points close to the origin. If we take Conjecture 1 as true, this leads consequently to $m$-morphism with a very large $m$ with the dominant traits being possibly very close to each other. This is therefore very difficult to capture numerically. Moreover, when $\e>0$ is large with respect to the distances between the dominant traits, the mutations may flatten the solution leading to a smooth unimodal distribution.
 
	
In the stochastic simulations in \cite{BilliardFerriereMeleardTran,CalvezFigueroaAl}, a Heaviside function was used for the transfer kernel. As discussed in the previous paragraphs, in our framework this would correspond to a degenerate case with $K\to \infty$ and it is difficult to compare our results with the numerical simulations in this case. Note also that in \cite{BilliardFerriereMeleardTran} an affine function was considered for the growth rate $R$ instead of a quadratic term in our work. Our investigations seem to indicate that this choice of growth rate would not modify significantly our qualitative results (results not shown).


	\subsection{Organization of the article}
	 
	 In Section \ref{sec:dirichlet} we study the existence of solutions and prove  Lemma \ref{lem:limit_eigenvalue} and Theorem \ref{thm:existence}. Section~\ref{section:limit_varepsilon} is devoted to the study of the vanishing diffusion limit and the proof of Theorem~\ref{thm:limit_epsilon}. In Section~\ref{section:limit_profile} we study the limit profiles and the monomorphic and dimorphic cases in detail, proving Theorem~\ref{thm:main}. Finally, there is a supporting section,  Annex~\ref{section:annex}, where we discuss the numerical schemes utilized to obtain the results presents in the numerical simulations.
	
	\section{Existence of solutions}
	\label{sec:dirichlet}

   In this section, we prove Lemma \ref{lem:limit_eigenvalue} and Theorem \ref{thm:existence}.
   
   \subsection{The proof of Lemma  \ref{lem:limit_eigenvalue}}


%
%
		
	
	Formula~\eqref{eq:eigenvalue} shows that if $\varepsilon_1>\varepsilon_2$ then $\lambda_{\varepsilon_1}(\R)\leq \lambda_{\varepsilon_2}(\R)$ and thus the following quantity
		$$
		\lim\limits_{\varepsilon\to 0}\lambda_{\varepsilon}(\R)
		$$
		is well defined. Let us call this limit $\lambda_0$ for the time being.

		
		Next, we can perform the Hopf-Cole transformation
		$$
		N_\varepsilon(z)=\exp\left(\frac{V_\varepsilon(z)}{\varepsilon}\right)
		$$
		on equation~\eqref{eq:eigenpair_whole_space} to arrive to an equation of the form
		$$
		-\varepsilon V''_\varepsilon(z) - ( V'_\varepsilon(z))^2 = R(z) -\lambda_\varepsilon(\R),\quad z\in\R.
		$$
		By arguments similar to the ones in Section~\ref{section:limit_varepsilon} we can take the limit as $\varepsilon\to 0$ to arrive to the problem
		$$
		\begin{cases}
			- ( V'_\varepsilon(z))^2  = R(z) - \lambda_0,\quad z\in\R,\\
			\max\limits_{z\in\R}V(z)=0.
		\end{cases}
		$$
		This already implies that $\lambda_0\geq R(z)$ for all $z\in\R$, but since a maximum is attained there must be a point where $\lambda_0= R(z)$. This two conditions can only be satisfied if $\lambda_0=1$. Since $\lambda_\varepsilon(\R)$ was increasing as $\varepsilon\to 0$ we conclude that $\lambda_\varepsilon(\R)\leq 1$.
		
	
%
%
	
	\subsection{The proof of Theorem \ref{thm:existence}}
	
	In order to study the problem in the whole space we will begin by focusing on the Dirichlet problem posed in a bounded domain $\Omega$. The problem reads
	\begin{equation}\label{eq:problem_dirichlet}
		\begin{cases}
			-\varepsilon^2 n''(z) = n(z)[R(z) - \rho(n)] + \tau n(z)\displaystyle\int_\Omega \frac{n(y)}{\rho}H(z-y)\ {\rm dy},\quad z\in\Omega,\\
			n(z)=0,\quad z\in\partial \Omega,\\
			n(z)>0,\quad z\in \Omega,\\
			\rho(n) = \displaystyle\int_\Omega n(z){\rm dz}.
		\end{cases}
	\end{equation}
	Note that, by assumption \eqref{eq:hypothesis_R}, $R\in C(\mathbb{R})$ is such that that there exists a bounded domain $ D_\tau\subset\mathbb{R}$ such that  $R(z)< -\tau$ for all $z\in D_\tau^c$. From now on, in these notes, we will consider only bounded intervals $\Omega$ such that for all $x\in D_\tau$, $x\pm 1\in \Omega$,  unless specified otherwise, though some of the results presented here are true for more general domains. Since we will make $\Omega\to\R$ this will suppose no limitation in our final result.
	
	It will also be convenient to study the eigenvalue problem in a bounded domain with Dirichlet condition:
	\begin{equation}\label{eq:eigenpair}
		\begin{cases}
			-\varepsilon^2  N''_\varepsilon(z) - R(z)N_\varepsilon(z) = -\lambda_\varepsilon(\Omega) N(z),\quad z\in\Omega\\
			N_\varepsilon(z)>0, \quad z\in \Omega\\
			N_\varepsilon(z)=0, \quad z\in\partial \Omega	.	
		\end{cases}
	\end{equation}
	We write $\lambda_\varepsilon(\Omega)$ to highlight the dependence on the domain of the eigenvalue. We remark that $\lambda_\varepsilon(\Omega)$ is increasing with respect to $\Omega$, and also that the eigenvalue $\lambda_\varepsilon(\Omega)$ can be characterized by Rayleigh quotients as
	\begin{equation}\label{eq:eigenvalue_dirichlet}
		-\lambda_\varepsilon(\Omega) = \min\limits_{u\in H^1_0(\Omega), u\neq 0}\frac{1}{\|u\|^2_{L^2(\Omega)}}\left\{ \int_\Omega\varepsilon^2 (u')^2- Ru^2 {\rm dz}\right\}.
	\end{equation}

	We will prove the existence of solutions via the Leray-Schauder's Theorem.\\ Let
	$$
	\rho(m):= \int_\Omega m,\qquad\phi(m):= \tau\displaystyle\int_\Omega \frac{m(y)}{\rho(m)}H(z-y)\ {\rm dy}\in [-\tau, \tau]
	$$
	and, for $\mu\in(0,1)$ and $m \in L^2(\Omega)$, we define the operator $\Gamma:(0,1)\times L^2(\Omega)\to L^2(\Omega)$ by $\Gamma(m,\mu)=n$ with $n$ the solution to the following problem
	\begin{equation}\label{eq:topological_degree}
		\begin{cases}
			-\varepsilon^2 n'' - Rn + n + \rho(m)n=  m + \mu\phi(m)m,\quad z\in\Omega\\
			n(z)=0,\quad z\in\partial \Omega.
		\end{cases}
	\end{equation}
Recall that $L^2(\Omega)\subset L^1(\Omega)$.    Note that by classic theory, $\Gamma$ is a continuous and compact operator.   Finally, let us highlight that $n$ satisfies
	$$
	-\varepsilon^2 n'' +a(z)n\geq 0 \text{ with }a(z)=-R(z)+1+\rho(m)\geq 0
	$$
	whenever $m\geq 0$, since $\tau< \lambda_\e(\Omega)\leq 1$. Therefore,  the maximum principle ensures that $n\geq 0$.
	
	Before showing existence of solutions via a topological degree argument, let us show that any fixed point of the operator $\Gamma$ satisfies an upper bound that will help us obtain $L^2$ norms from bounds for $\rho(n)$. 
	Let us define $z_0$ as $n(z_0)=\max\limits_{z\in\Omega} n(z)$.
	
	\begin{lem}\label{lem:upper_bound}
		The maximum of  any fixed point $n$ of $\Gamma$ is attained at a point $z_0\in D_\tau$. Moreover, we have that
		$$
		n(z_0)=\max\limits_{z\in\Omega} n(z)\leq \delta_2\rho(n)
		$$
		for all $\Omega$, where $\delta_2$ is a positive constant depending only on $\textcolor{red}{\e}$, the function $R$ and the constant $\tau$.
	\end{lem}
	\begin{proof}

		Let us see first that, under the assumptions on $R$, $z_0$ is localized for all $\Omega$.  We have that $ n''(z_0)\leq 0, n(z_0)\geq 0$ and thus, since $n$ is a fixed point of $\Gamma$,
		\begin{equation}\label{eq:z_0}
			R(z_0)-\rho(n) + \mu\tau\displaystyle\int_\Omega \frac{n(y)}{\rho(n)}H(z_0-y)\ {\rm dy}\geq 0
		\end{equation}
		which implies that $R(z_0)\geq \rho(n)-\mu\tau>-\tau$ and therefore $z_0$ must be contained in the domain $D_\tau$. This means that for all $\Omega\subset\mathbb{R}$ we have that $z_0\in D_\tau$.
		
		Next we shall find an upper bound for our fixed point. We notice that by Harnack inequality, there exists a constant $c$ such that $c n(z_0)\leq n(z)$ for all $z \in B_1(z_0)$ and $B_1(z_0)\subset\Omega$. Therefore
		$$
		n(z_0)=\int_{z_0-1}^{z_0+1}\frac{n(z_0)}{2}{\rm dz} \leq \int_{z_0-1}^{z_0+1}\frac{n(z)}{2c}{\rm dz}\leq \frac{\rho(n)}{2c}
		$$
		and thus $n(z_0)\leq \rho(n)/\delta$ with $\delta=2c$.
	\end{proof}
	
	Now we have the tools to show the existence of a solution to our problem.

	\begin{prp}\label{prp:bounded_domain}
		For all $\tau<\lambda_\varepsilon(\Omega) $ there exists a non-negative solution $n$ to problem~\eqref{eq:problem_dirichlet} satisfying
		\begin{equation}
		\label{ineq:lambdarho}
		0<\lambda_\varepsilon(\Omega)-\tau \leq \rho(n)\leq 1+\tau 
		\end{equation}
		for all $\Omega\subset\R$.
	\end{prp}
	\begin{proof}

		Let $\mathcal{X}$ be a bounded subdomain of $L^2(\Omega)$ such that every $u\in \mathcal{X}$ satisfies
		\begin{itemize}
			\item[(i)] $u>0$ for all $z\in \Omega$,
			\item[(ii)] $({\lambda_\varepsilon}(\Omega)-\tau)/2< \rho(u)<2(1+\tau)$ and 
			\item[(iii)] $\frac{\lambda_\varepsilon(\Omega)-\tau}{2|\Omega|}< \|u\|_{L^2(\Omega)}< 2\delta_2^{1/2}(1+\tau)$.
		\end{itemize}
		
		Since $\Gamma$ is a compact operator, in order to show existence of solutions by Leray-Schauder's Theorem we only need to check that the set of fixed points of $\Gamma:\mathcal{X}\to L^2(\Omega)$ for some $\mu$ (i.e. the set of functions satisfying $\Gamma(n,\mu)=n$) is bounded away from the boundary of $\mathcal{X}$, and also that for some value $\mu$, say $\mu =0$, the equation has a unique fixed point.
		
		This last condition is easily checked since for $\mu=0$ the operator $\Gamma$ admits a unique fixed point given by the eigenfunction $N$ of~\eqref{eq:eigenpair} with mass $\lambda_\varepsilon(\Omega)$ (i.e., from all the possible multiples of the normalized eigenfunction, the one with mass $\lambda_\varepsilon(\Omega)$).
		Let us check then the boundary condition.
		
		If $n$ is a fixed point of $\Gamma(n,\mu)$ the problem~\eqref{eq:topological_degree} becomes
		$$
		\begin{cases}
			-\varepsilon^2 n'' - Rn + \rho(n)n= \mu\phi(n)n,\quad z\in\Omega\\
			n(z)>0,\quad z\in \Omega\\
			n(z)=0,\quad z\in\partial \Omega
		\end{cases}
		$$
		Let us multiply the equation therein by any $N$ solution of~\eqref{eq:eigenpair}, and arrive to
		$$
		(\SM{-}\lambda_\e(\Omega) + \rho(n)) \int_\Omega nN -\mu\int_\Omega \phi(n)nN = 0
		$$
		which implies, since $\mu\phi(n)\in [-\tau,\tau]$ and $nN>0$ in $\Omega$, that
		\begin{equation}\label{eq:bound_mass}
			(\SM{-}\lambda_\varepsilon(\Omega) + \rho(n) + \tau)  \geq 0,
		\end{equation}
		meaning that $\rho(n)\geq \lambda_\varepsilon(\Omega) - \tau$ for all the fixed points of $\Gamma(\cdot,\mu)$. Since $n\geq 0$ and $\Omega$ is bounded, from Cauchy-Schwartz inequality we can also recover that
		$$
		\|n\|_{L^2(\Omega)}\geq \frac{\lambda_\varepsilon(\Omega)-\tau}{|\Omega|}.
		$$
		
		Let us check now the upper bound for the mass $\rho(n)$.  We take a $\varphi$ such that $\varphi(z)=0$ if $z\in\partial(\Omega)$ and $\varphi(z)>0$ and $\varphi''(z)<0$ if $z\in \Omega$. If we multiply the equation satisfied by the fixed points of $\Gamma$ by $\varphi$, integrate over $\Omega$ and then integrate by parts in the laplacian term we can see that
		$$
		\rho(n)\int_\Omega n\varphi \leq  \int_\Omega n\varphi + \tau \int_\Omega n\varphi,
		$$
		which immediately provides that $\rho(n)\leq 1+\tau$. Using Lemma~\ref{lem:upper_bound} we next deduce that
		$$
		\|n\|_{L^2(\Omega)}\leq \delta_2^{1/2}(1 + \tau).
		$$
		
		These computations show that for each $\mu\in(0,1)$ the fixed point $n$ are away from the boundary of $\mathcal{X}$. The Leray-Schauder's Theorem ensures then that there exists at least one fixed point for $\mu = 1$, this is, a solution to our problem.
	\end{proof}

	The next point in our study, in order to pass to the limit $\Omega\to \mathbb{R}$, is to obtain uniform bounds for the solution $n$ that do not depend on $\Omega$. The uniform upper bound is a direct consequence of Lemma~\ref{lem:upper_bound}. In order to find a lower bound for $n(z_0)$ (which is, remember, the maximum of the solution) we will find an estimate for the tails of the solutions that will keep the mass from "escaping to infinity", meaning that the mass must concentrate "at the center" of $\Omega$ when this domain is big, providing a lower bound for $n(z_0)$.
	
	\begin{prp}\label{prp:lower_bound_dirichlet}
		Let $n$ solve problem~\eqref{eq:problem_dirichlet} and let $\tau<1$. There exist an $\varepsilon_0\in (0,1)$ and a pair of positive values $z_*, \tilde c$ independent of $\varepsilon$ and $\Omega$ such that
		\begin{equation}
		\label{upper-bound}
		n_\varepsilon(z)\leq v(z):= (1+\tau)e^{-\frac{\tilde c}{\varepsilon}(|z|-z_*)}\text{ for all }|z|\geq z_*,\ \varepsilon<\varepsilon_0.
		\end{equation}
		As a consequence, we have that
		$$
		\max\limits_{z\in\Omega} n_\varepsilon(z)\geq  \tilde\delta(1-\tau),
		$$
		where $\tilde\delta$ is a positive constant depending only on the function $R$ and the constant $\tau$.
	\end{prp}
	
	\begin{proof}We will argue for $z\geq 0$, since for $z\leq 0$ it is analogous. Let us call the right boundary value of $D_\tau$ 
		, $z_{\tau,2}$. Since $n_\varepsilon(z)\to 0$ as $|z|\to \infty$ and $ n_\varepsilon''(z)\geq 0$ for all $z$ outside of $D_\tau$ we conclude that $n'_\varepsilon(z)\leq 0$ for all $z\geq z_{\tau,2}$. This means that for all $z\in K_1:=[z_{\tau,2}, z_{\tau,2}+1], n_\varepsilon(z)\geq n_\varepsilon(z_{\tau,2}+1)$, and thus
		$$
		1+\tau \geq \rho(n_\varepsilon)>\int_{K_1} n_\varepsilon>|K_1|\cdot n_\varepsilon(z_{\tau,2}+1)=n_\varepsilon(z_{\tau,2}+1).
		$$
		Since $n'_\varepsilon(z)\leq 0$ for all $z\geq z_{\tau,2}$ we conclude that $n_\varepsilon(z)\leq 1+\tau$ for all $z\geq z_{\tau,2}+1$. Note that this bound does not depend on $\varepsilon$.
		
		Since $\lambda_\varepsilon(\Omega)$ increases to $\lambda_\varepsilon(\R)$ as $\Omega$ tends to $\R$ and $\lambda_\varepsilon(\R)$ itself increases to 1 as $\varepsilon\to 0$, thanks to Lemma \ref{lem:limit_eigenvalue}, there must exist two constanst $\varepsilon_0,c_1\in(0,1)$ such that $\lambda_\varepsilon(\widetilde \Omega)\geq c_1$ for all $\varepsilon\in(0,\varepsilon_0)$  and all $\widetilde \Omega$ such that $\Omega\subseteq\widetilde \Omega\subseteq \R$. Therefore, using \eqref{ineq:lambdarho},$n_\varepsilon$  satisfies the inequality, for $\e\leq \e_0$,
		$$
		-\varepsilon^2 n''_\varepsilon - \left(R+2\tau - c_1\right)n_\varepsilon\leq 0,\quad z\in \R.
		$$
		We next define $\tilde c:=4(1+\tau)/(1-\tau)$ and choose $z_1$ such that
		$$
		-R(z)+c_1\geq \tilde{c}^2-2\tau,\qquad \text{for all $|z|\geq z_1$}.
		$$
		Note that Assumption \eqref{eq:hypothesis_R} guarantees that we can choose such constant $z_1$. We then deduce that, for $\e\leq \e_0$,
		$$
		-\varepsilon^2 n''_\varepsilon -  \tilde{c}^2n_\varepsilon\leq 0,\quad  \text{for all }z\geq z_1.
		$$
		We next define $v(z):=(1+\tau)e^{-\frac{\tilde c}{\varepsilon}(z-z_*)}$ which satisfies 
		$$
		-\varepsilon^2 v''_\varepsilon -  \tilde{c}^2v_\varepsilon=0.
		$$
		If we take now $z_* =\max\{z_1, z_{\tau,2}+1\}$ we get that $v$ is a supersolution in $(z^*,+\infty)$  for all $\varepsilon\in (0,\varepsilon_0)$  and $v(z_*)=1+\tau\geq n_\varepsilon(z_*)$. A comparison argument based on the maximum principle readily provides that $n_\varepsilon(z)\leq v(z)$ for all $z\geq z_*$. A similar argument works for $z\leq -z_*$, proving the first claim of our lemma.
		
		Note also that, for all $\e\leq 1$,
		$$
		\int_{z_*}^\infty v(z)dz = \frac{(1+\tau)\varepsilon}{\tilde c}\leq \frac{1-\tau}{4},
		$$ 
		and hence
		$$
		\int_{z_*}^\infty n_\e(z)dz \leq \frac{1-\tau}{4}.
		$$
		Similarly, one can show that   
		$$
		\int_{-\infty}^{z^*}  n_\e(z)dz \leq \frac{1-\tau}{4},
		$$
		and consequently
		$$
		\int_{-z_*}^{z_*}n_\varepsilon \geq  \frac{1-\tau}{2}.
		$$
	This can only hold if $n_\varepsilon(z_0)\geq (1-\tau)/4z_*$, with $z_0$ the maximum point of $n_\e$. We conclude by taking $\tilde\delta=(4z_*)^{-1}$.

	\end{proof}
%
%
	
	With all the work done in this section we can pass to the limit $\Omega_i\to\R$ in a sequence of domains $\{\Omega_i\}_{i=1}^\infty$ and obtain the first part of Theorem~\ref{thm:existence}. Let us denote, for a fixed $\varepsilon>0$, the solution of~\eqref{eq:problem_dirichlet} on each domain by $n_{i}$ (we omitted the sub-index $\varepsilon$ in this explanation for the sake of readability). By the previous results, we obtained uniform positive bounds for $\max n_i$ and $\rho(n_i)$  meaning that not only the family $\{n_i\}_{i=1}^\infty$, but also their respective second derivaties $ n_i''$, are uniformly bounded. We obtain by Ascoli-Arzelá's Theorem that the limit $n_i$ converges to  a non-trivial function $n$. Moreover, since we obtained a uniform exponential decay on the tails of the solutions and using the Dominated Convergence Theorem, we also deduce that $\rho(n_i)$ converges to $\rho(n)$. We then conclude from the elliptic regularity that $n$ satisfies	 \eqref{eq:main}.	
	
	Finally, in order to obtain equation~\eqref{eq:mass_varepsilon} we simply integrate all over $\R$ equation~\eqref{eq:main} to obtain, after bringing back the subindex $\varepsilon$ in the notation,
	$$
	  \displaystyle\int_\R R(z) n_\varepsilon(z)\ {\rm dz} - \rho(n_\varepsilon)^2 + \frac{\tau}{\rho(n_\varepsilon)} \int_\R \int_\R n_\varepsilon(z)n_\varepsilon(y)H(z-y)\ {\rm dy dz}=0
	$$
	but we notice that, thanks to $H$ being an odd function, the last integral term is equal to 0.
	%
	
	
	
	
	
	\section{The vanishing diffusion limit}\label{section:limit_varepsilon}
	
	Let
	$
	\rho_\varepsilon:=\rho(n_\varepsilon).
	$
	Replacing the Hopf-Cole transformed function~\eqref{eq:transformed_hopf_cole} into equation~\eqref{eq:main}	we obtain
	\begin{equation}\label{eq:hopf-cole}
		\begin{cases}
			-\varepsilon u''_\varepsilon(z) - \left| u'_\varepsilon(z) \right|^2 = R(z) - \rho_\varepsilon +\tau \displaystyle\int_\R \frac{n_\varepsilon(y)}{\rho_\varepsilon}H(z-y)\ {\rm d}y,\quad z\in\R,\\
			\rho_\varepsilon = \displaystyle\int_\R n_\varepsilon(z){\rm dz},
		\end{cases}
	\end{equation}
	
	Let us define
	$$
	\Phi_\varepsilon(z):=\tau \displaystyle\int_\R \frac{n_\varepsilon(y)}{\rho_\varepsilon}H(z-y)\ {\rm d}y.
	$$
	Note that thanks to Assumption \eqref{eq:hypothesis_H}, $\Phi_\e(z)\in [-\tau,\tau]$ and $\Phi'_\varepsilon(z)\in [0,\tau]$	
	
	For the study of equation~\eqref{eq:hopf-cole} we will need the following regularity estimates.
	\begin{prp}\label{prp:regularity} Let $u_\varepsilon$ be a solution of~\eqref{eq:hopf-cole}.
	\begin{itemize}
	\item[(i)] Let $D\subset \R$ be a bounded domain. Then there exists a positive constant $C(R,\tau, D)$ such that for all $\varepsilon\in (0,1)$,
		$$
		\left| u'_\varepsilon(z) \right|\leq C(R,\tau, D)\text{ for all }z\in D.
		$$
	
	\item[(ii)] Given any compact set $K$ there exists a positive constant $C$ independent of $\varepsilon$ such that $|u_\varepsilon(z)|\leq C$ for all $z\in K$.
	
	\item[(iii)] There exists a positive constant $S_c$ such that
		$$
		\partial^2_{z}u_\varepsilon(z)\geq -S_c\quad\text{for all}\quad\varepsilon\in (0,1).
		$$
		In other words, the family $u_\varepsilon$ is uniformly semi-convex.
	\end{itemize}
	\end{prp} 
	In the next subsection we will prove step by step each of these claims, and then after that, Theorem~\ref{thm:limit_epsilon}.
	
	\subsection{The regularity estimates on $u_\varepsilon$. Proof of Proposition~\ref{prp:regularity}}

	\textbf{Proof of Proposition~\ref{prp:regularity} (\textit{i})}:
		We start by choosing an open set $E$ such that $D\subset E\subset R$ and $\mathrm{dist}(D,\p E)>1$, and   a smooth, compactly supported cut-off function $\zeta$   such that
		$$
		\zeta(z)\equiv 1\text{ for all }z\in D\subset \supp(\zeta)=E
		$$
		and $|\zeta|<1$, $|\zeta'|_\infty + |\zeta''|_\infty\leq C $, where   $C$ is a positive constant. Let us also define $p(z):=(u_\varepsilon(z))'$. Replacing this in the equation on $u_\e$ we obtain the following equation on $p$:
		\begin{equation}\label{eq:p_gradient}
		-\varepsilon p'(z) -  (p(z))^2 = R(z) - \rho_\varepsilon + 	\Phi_\varepsilon(z),\quad z\in E.
		\end{equation}
		The next step is to differentiate the  equation above, to obtain
		$$
		-\varepsilon p''(z) -  2p(z)p'(z) = R'(z)+	\Phi'_\varepsilon(z),\quad z \in E.
		$$
		If we substitute the value of $p'(z)$ from~\eqref{eq:p_gradient} into the previous equation and then multiply it by $\varepsilon p(z)(\zeta(z))^4$ we find
		$$
		-\varepsilon^2 p''p\zeta^4 +2p^4\zeta^4+2(R-\rho_\varepsilon+\Phi_\varepsilon)p^2\zeta^4- \varepsilon( R'+	\Phi'_\varepsilon)p\zeta^4=0,
		$$
		where we omitted the $z$ variable for better eligibility.
		
		We next look  for a maximum of the function $w:=|p\zeta|$. This maximum should be attained at an interior point of $E$ that we denote by $z_m$ and, since the maximum of $w$ must coincide with the maximum of $w^2$, on such a point we must have
		$$
		[(p\zeta)^2]'(z_m)=0\Rightarrow p'(z_m)\zeta(z_m)=-p(z_m)\zeta'(z_m),
		$$
		and
		$$
		[(p\zeta)^2]''(z_m)\leq 0\Rightarrow p''(z_m)p(z_m)\zeta^2(z_m)\leq 2p^2(z_m)(\zeta')^2(z_m) - p^2(z_m)\zeta(z_m)\zeta''(z_m).
		$$
		Replacing this in the previous equation we obtain that, at the maximum point $z_m$,
		$$
		\varepsilon^2p^2\zeta^3\zeta'' - 2\varepsilon^2p^2\zeta^2(\zeta')^2 + 2p^4\zeta^4+2(R-\rho_\varepsilon+\Phi_\varepsilon)p^2\zeta^4- \varepsilon( R'+	\Phi'_\varepsilon)p\zeta^4\leq 0
		$$
		which means that
		$$
		2w^4(z_m)-\left[3\varepsilon^2C  + 2\big(C_1 +1+2\tau\big)\right] w^2(z_m) - \varepsilon\big(C_2 +\tau\big) w(z_m)\leq 0,
		$$
		where $C_1 =\max\limits_{z\in E}|R(z)|$ and $C_2 =\max\limits_{z\in E}|R'(z)|$. Note that we made use of the fact that $\rho_\varepsilon\leq 1+\tau$.
		
		This implies that there must exist a positive constant $k(R,\tau, D)$ such that
		$$
		 \sup\limits_{z\in E}|p\zeta|=|p\zeta| (z_m)\leq k(R,\tau, D).
		$$
		 We finish by comparing
		$$
		\sup\limits_{z\in D}|p|= \sup\limits_{z\in D}|p\zeta|\leq \sup\limits_{z\in E}|p\zeta|\leq k (R,\tau, E) .
		$$
		and thus
		$$
		|u'_\varepsilon|\leq k(R,\tau, E) \text{ for all }z\in D.
		$$
	
	\textbf{Proof of Proposition~\ref{prp:regularity} (\textit{ii})}:
	
		Let $z_{0,\varepsilon}$ be a point where $n_\varepsilon(z)$ attains its maximum. Note that, thanks to Lemma~\ref{lem:upper_bound}, $z_{0,\varepsilon}\in D_\tau$ for all $\varepsilon>0$ and thus there is no loss of generality in considering that $z_{0,\varepsilon}\in K$, since we can always argue in a bigger domain $\tilde K$ containing both $K$ and $z_{0,\varepsilon}$.
		
		Let us find a lower bound for $u_\varepsilon$. From Theorem~\ref{thm:existence} we have that
		$$
		u_\varepsilon(z_{0,\varepsilon})=\varepsilon \ln \big(n_\varepsilon(z_{0,\varepsilon})\big)\geq \varepsilon \ln \big(\delta(1-\tau)\big)\geq - \left|\ln \big(\delta(1-\tau)\big)\right|.
		$$
		This uniform lower bound for the function at a point and the uniform bound for the gradient in compact sets obtained in the previous proof allow us to conclude that there must exist a positive $C_1$ such that $u_\varepsilon(z)\geq -C_1$ for all $z\in K$.
		
		For finding the upper bound we recall that, by Proposition~\ref{prp:lower_bound_dirichlet}, we had  that $n_\varepsilon(z)\leq 1+\tau$ for all $z\geq z_{\tau,2}+1$. Note that this bound does not depend on $\varepsilon$. Once again,we can assume without loss of generality that $z_{\tau,2}+1\in K$ too. From here we obtain that
		$$
		u_\varepsilon(z_{\tau,2}+1)=\varepsilon \ln \big(n_\varepsilon(z_{\tau,2}+1)\big)\leq \varepsilon \ln \big(1+\tau\big)\leq |\ln \big(1+\tau\big)|
		$$
		and so, again from the bound for the gradient of $u_\varepsilon$ we deduce the existence of a positive constant $C_2$ such that $u_\varepsilon(z)\leq C_2$ for all $z\in K$.

\textbf{Proof of Proposition~\ref{prp:regularity} (\textit{iii})}:

		Similarly as we did when applying Bernstein's Method for finding uniform bounds for $u'_\varepsilon$, let us differentiate twice equation~\eqref{eq:hopf-cole} and multiply it by the square of a cut-off function $\varphi$ (with properties that will be specified later). If we rename $v:=u''_\varepsilon$ we obtain
		$$
		-\varepsilon v''\varphi^2-2v^2\varphi^2-2u'_\varepsilon v'\varphi^2-(R''+\Phi''(n_\varepsilon))\varphi^2=0.
		$$
		
		Let us define now $w:=v\varphi$. Clearly, at a point of minimum of $w$ (which must be attained inside the support of $\varphi$) we must have $w'=0$ and $w''\geq 0$. Rewriting   this   in terms of $v$ and $\varphi$ and substituting in the previous equation we obtain
		$$
		\varepsilon v\varphi''\varphi- 2 \varepsilon v(\varphi')^2-2v^2\varphi^2+2u'_\varepsilon
		v\varphi'\varphi-(R''+\Phi''(n_\varepsilon))\varphi^2\geq0.
		$$
		We can rearrange here, recalling that $R''+\Phi''(n_\varepsilon)$ is bounded by a positive constant $\tilde{C}$, to obtain
		$$
		2w^2\leq \left(\varepsilon \varphi'' - 2 \varepsilon \frac{(\varphi')^2}{\varphi}+ 2u'_\varepsilon\varphi'\right)w + \tilde{C}.
		$$
		
		Now we choose the properties of $\varphi$. Let us call again the left and right boundary values of $D_\tau$, $z_{\tau,1}$ and $z_{\tau,2}$ respectively. Since $n_\varepsilon(z)\to 0$ as $|z|\to \infty$ and $ n''_\varepsilon(z)\geq 0$ for all $z$ outside of $D_\tau$ we conclude that $n'_\varepsilon(z)\geq 0$ for all $z\leq z_{\tau,1}$ and $n'_\varepsilon(z)\leq 0$ for all $z\geq z_{\tau,2}$. Thus, the same can be said about $u'_\varepsilon$. Let us choose then $\varphi$ a positive function such that, for all $\varepsilon\in (0,1)$,
		\begin{equation}\label{eq:conditions_auxiliary_convexity}
			\begin{aligned}
			&\text{(i) }\varphi\equiv 1\text{ for all }z\in D_\tau\text{ and }\varphi\equiv 0\text{ for all }z\leq z_{\tau,1}-10\text{ and }z\geq z_{\tau,2}+10,\\
			&\text{(ii) } u'_\varepsilon\varphi'\geq 0\text{ for all }z\in\R,\\
			&\text{(iii) } \varepsilon \varphi'' - 2 \varepsilon \frac{(\varphi')^2}{\varphi}\in [-1,1].
			\end{aligned}
		\end{equation}
		Condition (ii) is equivalent with imposing that  $\varphi'(z)\geq 0$ for all $z\leq z_{\tau,1}$ and $\varphi'(z)\leq 0$ for all $z\geq z_{\tau,2}$, matching the sign of the derivative of $u_\varepsilon$ outside $D_\tau$. Let $z_1=\sup \{z\in \R : z<z_{\tau,1}, \varphi(z)=0\}$ and $z_2=\inf \{z\in\R : z>z_{\tau,2}, \varphi(z)=0\}$. Then condition (iii) would hold if $\varphi''$ and $\varphi'$ are small enough  and if    $\varphi(z)= c(z-z_i)^2+o(z-z_i)^2$ for all $z$ close to $z_i$, with $i=1,2$ and $c$ small enough. 
		
		With these properties in mind one can see that
		$$
		2w^2\leq f(z)w + \tilde{C}\quad\text{with}\quad f(z)\geq -1,
		$$
		meaning that $w\geq (-1-\sqrt{1+8\tilde{C}})/4$. We finish by making $\varphi$ converge to the constant 1 in the whole $\mathbb{R}$ and we do so by considering a sequence $\varphi_j:=\varphi\left(\frac{z}{j}\right)$  instead of $\varphi$ and making $j\to\infty$ in order to obtain a uniform bound in the whole $\R$ after noticing that all the $\varphi_j(z)$ satisfy~\eqref{eq:conditions_auxiliary_convexity} in bigger and bigger domains.

	\subsection{Proof of Theorem~\ref{thm:limit_epsilon}}
	
	The first step is passing to the limit as $\varepsilon\to 0$ in~\eqref{eq:hopf-cole}. As a consequence of Lemma~\ref{lem:limit_eigenvalue} and the uniform bounds for $\rho_\varepsilon$ we have that as $\varepsilon\to 0$ and along subsequences $\rho_\varepsilon$ converges to a constant $\rho_0$ such that
		$$
		0<1-\tau\leq \rho_0\leq 1+\tau.
		$$

	Now let us note that since $\Phi(n_\varepsilon)\in [-\tau,\tau]$, $\Phi'(n_\varepsilon)\in [0, \tau]$ and, in general, every $n^{th}-$derivative of $\Phi(n_\varepsilon)$ is uniformly bounded by a factors depending on $\tau$, then $\Phi(n_\varepsilon)$ converges, up to a subsequence $\varepsilon_{i_k}$, uniformly to a function
	$$
	\Phi_0:=\lim\limits_{i_k\to\infty}\Phi(n_{\varepsilon_{i_k}})\in C^\infty(\R).
	$$

	%
	With all the uniform bounds for $u_\varepsilon$ and $ |\nabla u_\varepsilon|$ the Ascoli-Arzelá Theorem provides the convergence along a subsequence in compact sets to a continuous function $u$. Semi-convexity comes from the uniform bound found in Proposition~\ref{prp:regularity}.
	
	Since $\Phi_0$ is a continuous, bounded function, the proof of $u$ being a viscosity solution of~\eqref{eq:viscosity_u} is analogous to the one in \cite{BarlesPerthame2} but in the time-independent case.
	
	The fact that $\max\limits_{z\in\R}u(z)=0$ comes from the observation that the integral
	$$
	\rho(n_\varepsilon)=\int e^{\frac{u_\varepsilon(z)}{\varepsilon}}dz
	$$
	is uniformly bounded from above and below away from $0$ thanks to \eqref{boundrho}. Note indeed that since $u_\e$ converges locally uniformly to  $u$, with $u$ a continuous function, $u$ cannot take positive values since otherwise $\rho(n_\varepsilon)$ would tend to $+\infty$ as $\e\to 0$. To prove that $u$ attains the value $0$, let's suppose by contradiction that $\max_z u(z)=-a<0$.Using the uniform convergence of $u_\e$ to $u$ and the upper bound \eqref{upper-bound} we deduce that
	$$
	\rho(n_\varepsilon)=\int e^{\frac{u_\varepsilon(z)}{\varepsilon}}dz\to 0,
	$$
	which is in contradiction with the lower bound in \eqref{boundrho}.

	Finally, since $n_\varepsilon\to 0$ whenever $u_\varepsilon < 0$ we conclude that 
	$$
	\supp n(z)\subseteq \{z\in\R : u(z)=0\}.
	$$
	Here, $\supp n$ is a set that must contain at least one point since $\rho_0>0$. Since $u$ is semi-convex it means that it is differentiable at its maximum points and thus $\nabla u$ must be equal to 0 when $u=0$. We then conclude thanks to \eqref{eq:viscosity_u} that
	$$
	\{z\in\R : u(z)=0\}\subseteq \left\{z\in\R : \rho_0-R(z)=\Phi_0(z)\right\}.
	$$
	
	%

	\section{Discussion of the limit profile}\label{section:limit_profile}
	
	From the results provided in the previous sections we deduce that, as $\e\to 0$, $n_\e$ converges to a measure $n$ such that $n$ corresponds to an evolutionary stable strategy (ESS), as stated in Proposition \ref{prop-ESS}. In this section we try to characterize the evolutionary stable strategies of the model, thereby describing the asymptotic behavior of the solutions.
	
	An evolutionary stable strategy as defined in Definition \ref{definition:ESS} may be composed of continuous subsets and singular points. We expect however that this set would be composed only of one or several isolated points such that the limit phenotypic density $n$ is a sum of Dirac masses. We believe also that $m$-morphic Evolutionary Stable Strategies with $m$ any integer number  may exist, provided the parameters of the model are chosen accordingly. However, we were unable to prove such a result in its general form.  In Section \ref{sec:supplimit} we  provide some conditions under which the only possible ESS is monomorphic and hence there is no ESS set with a continuous subset or non-isolated points. In Sections \ref{sec:mono} and \ref{sec:dimo} we provide necessary and sufficient conditions, respectively, for the existence of a  monomorphic and dimorphic ESS.  Moreover, we fully characterize the monomorphic or dimorphic ESS under their conditions of existence. Finally in Section \ref{sec:trimo} we discuss situations where the Evolutionary Stable Strategy may have more than two isolated points.

%
	
	From now on and for the rest of the article, we will focus on the case
	$$
	R(z)=1-gz^2.
	$$
	Let us also recall definition~\eqref{eq:mu}.
	
	\subsection{On the support of the limit}
	\label{sec:supplimit}
Let $n$ be a phenotypic density corresponding to an ESS. Thanks to Definition \ref{definition:ESS} $F(z)$ takes its maximum at the ESS points and hence
$$
	\begin{cases}
		\Phi_0(z) =\rho_0-1 + gz^2 \  & \text{for all}\  z\in\supp n,\\
		\Phi'_0(z)=2 gz \  & \text{for all}\  z\in\supp n,\\
		\Phi''(z) \leq 2g \  & \text{for all}\  z\in\supp n.
	\end{cases}
	$$
	Let us define
	$$
		C_2:=\left(\max\limits_{x\in\R}H''(x)\right)^{-1},
	$$
	and also the set
	$$
	Z:=\left\{z\in\R : R(z)-\rho_0+\Phi_0(z)=0\right\}.
	$$
	  We have then the following result.
	
	\begin{lem}
		Every point $z\in Z$ satisfies
		$$
		0\leq z\leq  \min\{\mu, 2\sqrt{\mu}\}.
		$$
		Moreover, if either $\mu< C_2$ or $\min\{\mu, 2\sqrt{\mu}\}\leq z_H$, where this $z_H$ comes from hypothesis~\eqref{eq:hypothesis_H}, then
		$$
		Z=\{z_0\}.
		$$
		In other words, the set $Z$ and, as a consequence, the support of $n$ consist in just one point.
	\end{lem}
	\begin{proof}
		
		Suppose that $z\in Z$. Since $\Phi_0(z)\in (-\tau,\tau)$, from $\rho_0-1 + gz^2 =\Phi_0(z)$ we deduce that
		$$
		-\tau <\rho_0-1 + gz^2<\tau.
		$$
		Using that $\rho_0\geq 1-\tau$ we deduce that $z\leq 2\sqrt{\mu}$. Similarly, since $\Phi'_0(z)\in (0,\tau]$ we deduce that
		$$
		0<2gz\leq\tau
		$$
		and $0<z\leq\mu$ follows.
		
		Next, let us suppose now that $Z$ has more than one point and call two of them $z_a$ and $z_b$. In such points we have that $\Phi'_0(z_a)=2gz_a$ and $\Phi'_0(z_b)=2gz_b$ and, since $\Phi_0 $ is smooth, by the Mean Value Theorem there must exist a point $z^*\in(z_a, z_b)$ such that
		$$
		\Phi''_0(z^*)=\frac{\Phi'_0(z_b)-\Phi'_0(z_a)}{z_b-z_a}=2g.
		$$
		Since, by the definition of $C_2$   we know that $\Phi''_0\in[-\tau C_2^{-1},\tau C_2^{-1}]$ we deduce that necessarily $\Phi''_0(z^*)=2g\leq\tau C_2^{-1}$, and thus $\mu\geq C_2$.
		
		Similarly, since $z_a<z^*<z_b$ and $\Phi''_0(z_a)\leq 2g = \Phi''_0(z^*)\geq \Phi''_0(z_b)$ and $\Phi_0\in C^\infty$ we deduce that there must exist a point $z^{**}\in(z_a,z_b)$ such that $\Phi'''_0(z^{**})=0$. This is to say that
		$$
		\int_{\R}   H''' (z^{**}-y)\ \frac{dn(y)}{\rho_0}  = 0,
		$$
		but this is only possible if $\sup\{Z\}>z_H$, since if not then $z_b\leq z_H$, meaning that $z^{**}<z_H$ and thus $z^{**}-y\in (-z_H,z_H)$ for all $y\in Z$ (remember that $y> 0$ for all $y\in Z$) and thus, by~\eqref{eq:hypothesis_H}, $H'''(z^{**}-y)<0$ for all $y\in Z$, making $\Phi'''_0(z)<0$, a contradiction. We conclude that, if there are more than one point in $Z$ then necessarily $z_H<\min\{\mu, 2\sqrt{\mu}\}$.
	\end{proof}

	\subsection{The monomorphic case}
	\label{sec:mono}
	In this section, we will provide necessary and sufficient conditions for the existence of a monomorphic ESS.
	Let us start the study of this case by commenting on an equation that will provide very useful insight into the monomorphic and the dimorphic cases. 
	
	\begin{lem}
	\label{lem:d1}
		Under the assumptions~\eqref{eq:hypothesis_H}, the equation
		\begin{equation}\label{eq:tanh-cosh}
			z \left( 1 + H'(z) \right) = 2 H(z),
		\end{equation}
		admits three solutions, which are $0$ and two constants $d_1$ and $-d_1$. Moreover $d_1\geq z_H$.
	\end{lem}
	\begin{proof}
		Let us define
		\begin{equation}
		\label{def:G}
		G(z)=2H(z)-z(1+H'(z)).
		\end{equation}
		Then it is easy to see that $G(0)=G'(0)=G''(0)=0$ and $G(z)\to -\infty$ as $z\to +\infty$, while, by the hypothesis on $H'''$,  $G''(z)\geq 0$ for all $z\in(0, z_H]$ while $G''(z)<0$ for all $z>z_H$. This proves that there exists a unique positive constant $d_1>z_H$ such that $G(d_1)=0$ and $G'(d_1)<0$. Since $H$ and hence $G$ are odd, we deduce that the zeros of $G$ are given by 0, $d_1$ and $-d_1$.
	\end{proof}
	
	 Another important value that will appear later on is
	\begin{equation}\label{eq:C_1}
	C_1 = 1-H'(d_1).
	\end{equation}
	Note that these values do not depend on any of the constitutive parameters of our problem but on the choice of the transfer function. With this in mind, let us define
	\begin{equation}\label{eq:mu_1}
	\mu_1=\frac{d_1}{C_1}.
	\end{equation}
	
	
	\begin{thm}
	\label{thm:ESSmon}
		Suppose the transfer kernel $H$ satisfies hypothesis~\eqref{eq:hypothesis_H}. Then there exists a monomorphic ESS if and only if
		\begin{equation} \label{eq:cond:mono}
			 \mu\leq \mu_1\quad\text{and}\quad \tau<\frac{2}{\mu}.
		\end{equation}
		Moreover, in this case we have 
		$$
		z_0 = \mu \qquad \text{and}\qquad \rho_0=1-\frac{\tau\mu}{2}.
		$$
	\end{thm}
	
	\begin{proof}
	
	We begin by assuming that there exists a monomorphic ESS $z_0$ such that
		 $$
		n=\rho_0\delta(z-z_0),\qquad \Phi_0(z)=\tau H(z-z_0), \qquad F(z)=1-gz^2-\rho_0+\tau H(z-z_0).
		$$
		Since $F(z)$ takes its maximum at the ESS point $z_0$ we deduce that
		\[ \begin{cases}
			1 - gz_0^2 - \rho_0 =0  \\
			-2gz_0 + \tau=  0.
		\end{cases} \]
		It follows that 
		$$
		z_0 = \frac{\tau}{2g}=\mu,\qquad \text{and}\qquad \rho_0 = 1 - \frac{\tau^2}{4g}=1-\frac{\tau \mu}{2}.
		$$
		Since $\rho_0>0$ is required, this imposes $\tau<2/\mu$.	It remains to prove that this candidate is an ESS if and only if $  \mu\leq \mu_1$. Note from Definition \ref{definition:ESS} that $z_0$ is an ESS point iff 
		$$
		\max_z F(z)= g\mu^2+\max_z  g\big(-z^2 +2\mu H(z-z_0)\big)=g\mu^2+g\,\max_z   J_{1,\mu}(z)=F(z_0)=0,
		$$	
		with
		$$
		J_{1,\mu}(z)=-z^2 +2\mu H(z-z_0).
		$$
		When $\mu=0$ the function $J_{1,\mu}$ is strictly concave and the  equality above trivially holds. Note also that thanks to Assumption \eqref{eq:hypothesis_H}, for all $\mu$, 
	$$
	J_{1,\mu}''(z_0)=-2<0.
	$$	
		Since $J_{1,\mu}$ is continuous with respect to $\mu$ and since $H(z-z_0)$ is strictly convex for  $z<z_0$, we deduce that there exists a constant $\mu^*$ such that for all $\mu<\mu^*$ the function $J_{1,\mu}$ attains its maximum at the only point $z_0$ and such that $J_{\mu^*}$ attains its maximum at least at two points: $z_0$ and a second point $z_1$. 
		We will prove that $\mu^*=\mu_1$. To this end, we use the fact that since $F$ takes its maximum at a second point $z_1$ we have
			$$
		\begin{cases}
			2H(z_1-\mu^*)=\frac{z_1^2-{\mu^*}^2}{\mu^*}\\
			z_1=\mu^* H'(z_1-\mu^*).
		\end{cases}
		$$ 
		which is equivalent to
		$$
		\begin{cases}
			2H(z_1-\mu^*)=\frac{z_1^2-{\mu^*}^2}{\mu^*}\\
			(z_1-\mu^*)\left(1+H'(z_1-\mu^*)\right)=\frac{z_1^2-{\mu^*}^2}{\mu^*}.
		\end{cases}
		$$
We then use Lemma \ref{lem:d1} to deduce that the only possibilities for $z_1$ are $z_1=\mu^*+d_1$ and $z_1=\mu^*-d_1$. Note that $z_1=\mu^*$ is excluded since we have supposed that $z_1\neq z_0$. The option $z_1=\mu^*+d_1$ leads to 
	$$
		\begin{cases}
			z_1^2={\mu^*}^2+2\mu^*H(d_1)\\
			z_1= \mu^* H'(d_1),
		\end{cases}
		$$
		which simply cannot hold (we would be getting that $z$ is at the same time bigger and smaller than $\mu^*$). We are left then with only one  option $z_1=\mu^*-d_1$, which yields
		$$
		\begin{cases}
			z_1^2={\mu^*}^2-2\mu^*H(d_1)\\
			z_1=\mu^* H'(d_1),
		\end{cases}
		$$
		and thus
		$$
		\mu^*=\frac{2H(d_1)}{1-(H'(d_1))^2}=\frac{d_1\left(1+H'(d_1)\right)}{\left(1-H'(d_1)\right)\left(1+H'(d_1)\right)}=\frac{d_1}{C_1}=\mu_1.
		$$
		This shows that $0\leq \mu\leq \mu_1$ implies the existence of a monomorphic ESS, provided $\tau<\frac 2\mu$. It remains to prove that for $\mu>\mu_1$, $\{z_0\}$ is not an ESS. To prove this, we will argue by contradiction. Let us suppose that $\{z_0\}$ is an ESS, which means, as we have seen, that
        $$
        F(z)=-g(z^2-\mu^2)+2g\mu H(z-\mu).          
        $$
        Then to arrive to a contradiction it is enough to show that for $\mu>\mu_1$
		$$
		\max\limits_{z\in\R}F(z)\geq  F (\mu - d_1) > 0=F(z_0).
		$$
        In order to do so let us define
        $$
        \tilde{f}(\mu):=\frac{F(\mu-d_1)}{g}=-d_1^2+2\mu(d_1-H(d_1))
        $$ and notice that, by definition, $\tilde{f}(\mu_1)=F (\mu_1 - d_1)/g=0$, while
        $$
        \frac{\partial\tilde f}{\partial \mu} =2(d_1-H(d_1))>0.
        $$
	\end{proof}
	
	It is important to note that the key point in this proof is that when the solution is monomorphic we are able to find an explicit formula for $F_1$. The same idea will be helpful when studying the dimorphic case.	
	
	As stated in \textrm{Conjecture} 1. we believe that to have $m$-morphic ESS $\mu$ should satisfy
	$$
	\mu_{m-1}\leq \mu \leq \mu_m,
	$$
	with $(\mu_i)$ an increasing sequence. One then could ask for instance what would happen if 
	$$
	0\leq \mu\leq \mu_{1}, \qquad \tau \geq \frac{2}{\mu} \,?
	$$
	Thanks to Theorem \ref{thm:ESSmon} we know that in this case, there does not exist any monomorphic ESS. From \textrm{Conjecture} 1. neither we  expect to have $m$-morphic ESS with $m>1$. The only remaining possibilities are the non-existence of ESS in this range of parameters (hence non-existence of a positive steady solution to \eqref{eq:main_growth} with $\e$ small) or the existence of an ESS which is not composed only of isolated points with possible continuous support. We believe the first option to hold that is the non-existence of  ESS. To verify  whether the existence conditions of Theorems~\ref{thm:existence} and~\ref{thm:limit_epsilon} are compatible with this hypothesis,
 one could ask if    the condition $\tau<1$ implies $\tau<2/\mu$.  As we will see in the next Lemma, this is the case since for all the transfer kernels considered in this work the value $\mu_1$ is necessarily smaller than 2, meaning that for all $\mu\in[0,\mu_1]$, the value $2/\mu$ is greater than 1 and thus any $\tau$ satisfying $\tau<1$ satisfies automatically $\tau<2/\mu$.
	
\begin{lem}\label{lem:mu_1_less_2}
	For all transfer kernel $H$ satisfying hypothesis~\eqref{eq:hypothesis_H} we have that $\mu_1<2$.
\end{lem}	
\begin{proof}
	Let us argue by contradiction and suppose that $\mu_1\geq 2$. Recall that 
	$$
	\mu_1=\frac{d_1}{1-H'(d_1)},\qquad \text{where $d_1$ is the unique positive solution of \eqref{eq:tanh-cosh} and $d_1>z_H$}.
	$$
	The condition $\mu_1\geq 2$ implies   that
	\begin{equation}\label{eq:cond_derivative_d_0}
		H'(d_1)\geq 1-\frac{d_1}{2}.
	\end{equation}
	
	Let us define   the polynomial
	$$
	P(z):=z-\frac{1}{4}z^2.
	$$
	This polynomial is important because it is the solution to the ODE
	$$
	z(1+P'(z))=2P(z)\quad\text{for all }z\in\R
	$$
	that satisfies that $\max\limits_{z\in\R} P(z) =P(2)= 1$.
	
	Now  recall the function $G$ defined in \eqref{def:G}.
	Thanks to Lemma \ref{lem:d1}, the function $G(z)$ has only one positive zero at the point $z=d_1$ with $G'(d_1)<0$. This means that
    $$
    d_1(1+H'(d_1))=2H(d_1),
    $$
    from where we deduce that necessarily $d_1<2$, since $H(d_1)<1$. Introducing~\eqref{eq:cond_derivative_d_0} in this last equation yields $H(d_1)\geq P(d_1)$.
    
    Let us see now that there exists a $\delta>0$ such that $H(z)> P(z)$ for all $z\in(d_1,d_1+\delta)$. If $H(d_1)>P(d_1)$ this is trivially true by continuity. Suppose on the other hand that $H(d_1)= P(d_1)$, meaning since $G(d_1)=0$ that $H'(d_1)=1-\frac{d_1}{2}= P'(d_1)$. Since $G'(d_1)<0$ this implies that $H''(d_1)>-\frac{1}{2}=P''(d_1)$ and thus the function $H(z)$ must be above $P(z)$ in a neighbourhood of the point $z=d_1$.

    Two options are now possible. Either $H(z)\geq P(z)$ for all $z\in [d_1+\delta, 2]$, which would led to a contradiction since $P(2)=1$ and $H(z)<1$, or there exists a point $z_p\in [d_1+\delta,2)$ such that
    $$
    \begin{cases}
        H(z_p)=P(z_p)=z_p-\frac{z_p^2}{4},\\
        H'(z_p)\leq P'(z_p)=1-\frac{z_p}{2},
    \end{cases}
    $$
    but this would mean that $G(z_p)\geq 0$, which is a contradiction with the fact that the only positive zero of $G(z)$ was $d_1$. We deduce that necessarily
	$$
	\mu_1=\frac{d_1}{1-H'(d_1)}<2.
	$$
\end{proof}

	\subsection{The dimorphic case}
	\label{sec:dimo}
		In this section, we first identify in a unique way the unique dimorphic ESS assuming its existence. Then we provide an additional assumption which allows us to provide a necessary and sufficient condition for the dimorphic ESS to exist.
\begin{lem}
\label{lem:dim}
	Let's suppose that there exists a phenotypic density $n$ corresponding to a dimorphic ESS such that
	 \[ n(z)  =  a  \delta_{z_1}(z) + b \delta_{z_2} (z),\qquad a+b=\rho_0, \qquad z_1>z_2.\]
	 Then, 	 
		\begin{equation}\label{eq:R-rho}
			\rho_0=  1-gz_1^2+g(\mu-\mu_1)H(d_1),
		\end{equation}
		and 
		\[  \begin{aligned}
			&z_1 = \mu\left(1-\frac{C_1}{2}\left( 1-\frac{\mu_1}{\mu}\right) \right),&& \qquad a = \frac{\rho_0}{2}\left( 1+\frac{\mu_1}{\mu}\right), \  \\
			&z_2 =\mu\left(1-\frac{C_1}{2}\left( 1+\frac{\mu_1}{\mu}\right) \right), && \qquad b = \frac{\rho_0}{2}\left( 1-\frac{\mu_1}{\mu}\right). 
		\end{aligned}\]
		Finally, the distance between the two values $z_1$ and $z_2$ is constant and equal to
		$$
		z_1-z_2=d_1
		$$
		Alternatively, equivalent formulas for the ESS points are
		$$
		z_1  = \mu \left(1 - \frac{C_1}{2}\right) + \frac{d_1}{2},\quad z_2  = \mu \left(1 - \frac{C_1}{2}\right) - \frac{d_1}{2}.
		$$
\end{lem}	
Considering such a phenotypic density we obtain
		\[  
			\Phi_0(z)  = \frac{a\tau}{\rho_0}H(z - z_1) + \frac{b\tau}{\rho_0} H(z - z_2) ,
		 \]
	and the corresponding fitness function
	\begin{equation}\label{eq:F_2_mu}
	\begin{array}{rl}
	F (z) &= 1-gz^2 - \rho_0 + \frac{\tau a}{\rho_0} H(z -z_1) + \frac{\tau b}{\rho_0} H(z -z_2)\\
	&= g\big(z_1^2-z^2-(\mu-\mu_1)H(d_1)+2\mu(\frac{a}{\rho_0}H(z-z_1)+\frac{b}{\rho_0}H(z-z_2))\big),\\
	&=: gJ_{2,\mu}(z).
	\end{array}
	\end{equation}
Note that
$$
J_{2,\mu_1}(z)=\mu_1-z^2+2\mu_1H(z-\mu_1)=\mu_1+J_{1,\mu_1}.
$$
Following the arguments in the proof of Theorem \ref{thm:ESSmon} we know that $J_{1,\mu_1}$ and hence $J_{2,\mu_1}$ attain their maximums at exactly two points. Moreover the value of $J_{2,\mu_1}$ at those maximum points is equal to $0$. We also note that for all $\mu>\mu_1$, $J_{2,\mu}$ takes the value $0$ at least at the  points $z_1$ and $z_2$. 
We then define
$$
\mu_2:=\inf \{\mu\geq \mu_1 \,|\, \text{$J_{2,\mu}$ takes the value $0$ at least at three points.}  \}.
$$

First, let us see that $\mu_2>\mu_1$, because, in principle, it could happen that after achieving the second maximum in the monomorphic case when $\mu=\mu_1$ then immediately a third maximum appears when $\mu>\mu_1$. In order to avoid this, it is enough to see that $J_{2,\mu_1}$ is strictly concave at its maxima, meaning, by the continuous dependence of $J_{2,\mu}$ on the parameter $\mu$, that in a small right-neighbourhood of $\mu_1$ the function $J_{2,\mu}$ is still strictly concave at its maxima and thus has only two of them. We readily compute from~\eqref{eq:F_2_mu} that $J_{2,\mu_1}''(z_1)=-2$, and on the other point 
$$
J_{2,\mu_1}''(z_2)=-2-2\mu_1H''(d_1)=-2-2\frac{d_1}{1-H'(d_1)}H''(d_1)=\frac{2}{1-H'(d_1)}G'(d_1)<0,
$$
since $G'(d_1)<0$ and $1-H'(d_1)>0$.

Let us see now that $\mu_2$ is finite. Notice that $z_1>z_2>0$ for all $\mu>\mu_1$. Since $J_{2,\mu}(0)\to \infty$ as $\mu\to \infty$ (due to the fact that $z_1^2$ behaves like $\mu^2$) it is clear then that $\mu_2<\infty$ and that there must exist at least a third point $z_3$ such that $J_{2,\mu}(z_3)=0$ when $\mu=\mu_2$.
  We are now ready to state our additional assumption

\noindent\textbf{Extra hypothesis on $H$ in the dimorphic case. } We then assume that
	\begin{equation}\label{eq:extra_hypothesis_H}
		\frac{\partial J_{2,\mu}}{\partial \mu}(z_3)>0\quad\text{for all }\mu\geq \mu_2.
	\end{equation}

	We want to remark again that, while cumbersome to write down for a general $H$, this condition is easily checked in the practice for many specific transfer kernels thanks to the fact that the formula for $J_{2,\mu}$ is going to be explicit and dependant only on $z,\mu$ and $H$. Let us show this in the particular case $H(z)=\tanh(z)$. In this case a numerical estimation of $\mu_2$ and $z_3$ is
	$$
	\mu_2\sim 4,03729,\quad z_3\sim 0,513.
	$$
	Having this we simply compute, thanks to the formula~\eqref{eq:F_2_mu}, that
	$$
	\frac{\partial J_{2,\mu}}{\partial \mu}(z_3)>1,148 \mu-2,675>0
	$$
	whenever $\mu\geq\mu_2$.
	
	\begin{thm}
	\label{thm:dim}
		Suppose the transfer kernel $H$ is such that the hypothesis~\eqref{eq:hypothesis_H} and~\eqref{eq:extra_hypothesis_H} are satisfied. Then there exists a positive constant $\mu_2$ depending only on $H$ such that there exists a dimorphic ESS if and only if
		\begin{equation} \label{eq:cond:di}
			\mu_1<\mu\leq \mu_2\quad\text{and}\quad \tau<\tau_2,
		\end{equation}
		with
		\begin{equation}
		\label{cond:tau2}
		\tau_2=\frac{2\mu}{z_1^2-(\mu-\mu_1)H(d_1)}.
		\end{equation}
	\end{thm}

	\begin{proof}[Proof of Lemma \ref{lem:dim}]
	Let us suppose that there exists a dimorphic ESS
	\[ n(z)  =  a  \delta_{z_1}(z) + b \delta_{z_2} (z),\qquad a+b=\rho_0.\]
Applying the ESS condition given in Definition \ref{definition:ESS} we obtain that
		\[\begin{aligned}
			&\big(  R - \rho_0 + \frac{a\tau}{\rho_0}H(\cdot - z_1) + \frac{b\tau}{\rho_0} H(\cdot- z_2) \big) (z_i) = 0, \\
			\text{and} \qquad & \big(  R - \rho_0 + \frac{a\tau}{\rho_0}H(\cdot - z_1) + \frac{b\tau}{\rho_0} H(\cdot- z_2) \big) '(z_i) = 0.
		\end{aligned} \]
		for $i=1,2$. We deduce that
		\begin{align}
			&1 - gz_1^2 - \rho_0 + \frac{ b}{\rho_0}  \tau H(z_1 -z_2) = 0   \label{eq:dimor1}\\
			& 1 - gz_2^2  - \rho_0 + \frac{ a}{\rho_0}  \tau H(z_2 -z_1)= 0  \label{eq:dimor2} \\ 
			&  -2gz_1 + \frac{a\tau}{\rho_0} + \frac{b\tau }{\rho_0}\cdot H'(z_1-z_2)  = 0,\label{eq:dimor3} \\
			&  -2gz_2 + \frac{b\tau}{\rho_0} + \frac{a\tau }{\rho_0}\cdot H'(z_2-z_1)=0.  \label{eq:dimor4}
		\end{align}

		Let us prove now that $z_1 - z_2 = d_1$. We first provide a new expression of $z_1 + z_2$ by adding \eqref{eq:dimor3} to \eqref{eq:dimor4}. Since $a + b = \rho_0$ and $H'$ is an even function, it follows
		\begin{equation}\label{eq:dimor:int1}
			z_1 + z_2 = \frac{\tau }{2g}\left(1 + H'(z_1-z_2) \right) .
		\end{equation}
		Next, we subtract \eqref{eq:dimor1} from \eqref{eq:dimor2} and recall that $H$ is an odd function to find 
		\[g(z_1 - z_2)(z_1 + z_2) -\tau H(z_1-z_2) =0. \]
		Substituting $z_1+z_2$ from~\eqref{eq:dimor:int1} into this last equation, recalling that $z_1>z_2$ and Lemma \ref{lem:d1}, yields
		$$
		z_1-z_2=d_1.
		$$
		Now we recover the expression for $a$ and $b$. Recall that
		$$
		C_1=1-H'(d_1).
		$$
		By subtracting the equation \eqref{eq:dimor4} from \eqref{eq:dimor3}, it follows
		\[-2g d_1 +  (a - b)  \frac{\tau}{\rho_0}C_1 = 0. \]
		From here we deduce that $a$ and  $b$ are solutions to the following system 
		\[ \left\lbrace
		\begin{aligned}
			&  a + b = \rho_0  \\
			&  a - b = \frac{2\rho_0 d_1 g}{\tau C_1} = \rho_0\frac{\mu_1}{\mu}.
		\end{aligned}\right. \]
		The expressions of $a,\ b$ follow. 		
		Finally, we recover the expression for $z_1$ and $z_2$. Since we know the values of $z_1 - z_2 = d_1$, $a$ and $ b$  we can deduce with the help of \eqref{eq:dimor3} and \eqref{eq:dimor4} the expressions for $z_1$ and $z_2$, which must be the same as the values obtained from \eqref{eq:dimor1} and \eqref{eq:dimor2}. From this last equality condition we will deduce a formula for $\rho_0$.
	
		From \eqref{eq:dimor3} and \eqref{eq:dimor4} we deduce
		$$
		z_1=\frac{\mu }{\rho_0}\left(\rho_0 - bC_1 \right),\quad z_2=\frac{\mu }{\rho_0}\left(\rho_0 - aC_1 \right),
		$$
		from which the values of $z_1,z_2$ stated in the theorem follow. Furthermore,   from \eqref{eq:dimor1} and \eqref{eq:dimor2} we obtain
		$$
		z_1^2=\frac{1-\rho_0}{g} + \frac{b}{g\rho_0}\tau H(d_1) ,\quad z_2^2=\frac{1-\rho_0}{g} - \frac{a}{g\rho_0}\tau H(d_1).
		$$
		From the previous equations we deduce that  
		$$
		\begin{aligned}
			\frac{\mu^2 }{\rho_0^2}\left( \rho_0-bC_1 \right)^2 = \frac{1 - \rho_0}{g}+\frac{b\tau H(d_1)}{g\rho_0},\\
			\frac{\mu^2 }{\rho_0^2}\left( \rho_0-aC_1 \right)^2 = \frac{1 - \rho_0}{g}-\frac{a\tau H(d_1)}{g\rho_0}.
		\end{aligned}
		$$
		 Adding both equations provides
		$$
		\frac{\mu^2 }{\rho_0^2}\left( 2\rho_0^2 + (a^2+b^2)C_1^2   -2\rho_0^2C_1 \right) = \frac{2(1 - \rho_0)}{g}-(a-b)\frac{\tau H(d_1)}{g\rho_0}.
		$$
		Since
		$$
		a-b=\rho_0\frac{\mu_1}{\mu}\quad \text{and}\quad a^2+b^2=\frac{\rho_0^2}{2}\left(1+\frac{\mu_1^2}{\mu^2}\right)
		$$
		we obtain
		$$
		\mu^2 \left( 2 + \frac{C_1^2}{2}\left(1+\frac{\mu_1^2}{\mu^2}\right)   -2C_1 \right) = \frac{2(1 - \rho_0)}{g}-2\mu_1 H(d_1)
		$$
		and thus we readily see that
		$$
		\frac{1 - \rho_0}{g}= \mu_1 H(d_1) + \frac{\mu^2 }{2}\left( 2 + \frac{C_1^2}{2}\left(1+\frac{\mu_1^2}{\mu^2}\right)   -2C_1 \right),
		$$
		and from here we can deduce the equation for $\rho_0$.		\end{proof}


	\begin{proof}[Proof of Theorem \ref{thm:dim}]
	To prove Theorem \ref{thm:dim} it is enough to prove that the parameters obtained in Lemma \ref{lem:dim} provides us with a dimorphic ESS if and only if \eqref{eq:cond:di}  holds. It is easy to verify that
	the properties 
	$$
	0<\rho_0,\qquad 0< a_0-b_0< \rho_0,
	$$
	imply that
	$$
	\mu_1< \mu, \qquad \text{and} \qquad \tau<\tau_2.
	$$
	
		To check the $\mu\leq \mu_2$ we recall that $J_{2,\mu_2}(z)$ takes the value $0$ at a third point $z_3$ and that thanks to	 the hypothesis~\eqref{eq:extra_hypothesis_H}, we have that $\frac{\partial J_{2,\mu}}{\partial \mu}(z_3)>0$ for all $\mu\geq \mu_2$.
		Therefore, for all $\mu>\mu_2$ we have that $J_{2,\mu_2}(z_3)>0$, {implying that the corresponding phenotypic density does not correspond to an ESS.} We have shown that dimorphism implies $\mu_1<\mu\leq \mu_2$.
		Finally, from the definition of $\mu_2$ and Definition \ref{definition:ESS}, it is immediate that the phenotypic density obtained in Lemma \ref{lem:dim} corresponds to an ESS as long as $\mu_1< \mu\leq \mu_2$ and $\tau<\tau_2$.	\end{proof}
	
	
	\subsection{The trimorphic and $m$-morphic cases}
	\label{sec:trimo}
	
	We expect that in the range of parameters provided in Conjecture 1, there would exist an $m$-morphic ESS. The computations become more involved for $m\geq 3$ and we were not able to obtain an analytic result for  these cases. Nevertheless, our numerical explorations seem to confirm Conjecture 1. Here, we discuss briefly the trimorphic case, where the corresponding phenotypic density   is given by
	$$
	n(z)=\sum\limits_{i=1}^3 a_i\delta_{z_i}(z),\qquad a_1+a_2+a_3=\rho_0,
	$$
	so that the transfer term becomes
	$$
	\Phi_0(z) =\tau \Psi_0(z)= \frac{a_1 \tau}{\rho_0}H(z - z_1) + \frac{a_2\tau}{\rho_0} H(z - z_2)  + \frac{a_3\tau}{\rho_0} H(z - z_3).
	$$
	Similarly to the computations in the previous sections, the fact that $n$ corresponds to an ESS implies that
	$$
	1 - gz_i^2 - \rho_0 +  \tau \Psi_0(z_i) = 0,\qquad -2gz_i+\tau\Psi_0'(z_i), \qquad i=1,2,3, 
	$$
	which can be written as 
	\begin{equation}
	\label{eq:zi-tri}
	-z_i^2+2\mu  \Psi_0(z_i)=\frac{\rho_0-1}{g},\qquad -z_i+\mu \Psi_0'(z_i)=0,\qquad i=1,2,3.
	\end{equation}
		We   devised a program to compute the roots of this system  in the form of the triplets $(z_1, z_2, z_3)$ and $(a_1, a_2, a_3)$ and the mass $\rho_0$, in the particular case $H(z)=\tanh(z)$. It was done through Newton's method. In view of the structure of the system \eqref{eq:zi-tri} we expect that the quantities $(z_1, z_2, z_3)$ and $(a_1, a_2, a_3)$ would depend only on the parameter $\mu$ but that the total size $\rho_0$ would depend on both on $\mu$ and $g$.  We managed to characterize the trimorphic case all the way from $\mu_2$ to a quite precise candidate for $\mu_3$, as shown in  Table \ref{table-tri}. Note that \eqref{eq:zi-tri} is not a sufficient condition for $n$ to correspond to an ESS. We however verified numerically that the obtained parameters lead to a fitness function $F$ which takes its maximum at the points $z_i$, $i=1,2,3$, and hence the results correspond to evolutionary stable strategies. In Figure \ref{fig:tri} we illustrate a numerical solution of  \eqref{eq:time_dependant_main}, with parameters $\mu=5$ and $g=0.05$, where we observe that in long time the solution indeed converges to a trimorphic phenotypic density close to the one obtained in Table \ref{table-tri} for the case $\mu=5$.
	
	\begin{table}[htb]
		\centering
		\begin{tabular}{|l|l|l|l|l|l|l|l|l|}
			\hline
			\rowcolor[HTML]{96FFFB} 
			\textbf{$\mu$} & \textbf{$g$} & \textbf{$z_1$} & \textbf{$z_2$} & \textbf{$z_3$} & \textbf{$\frac{a_1}{\rho_0}$} & \textbf{$\frac{a_2}{\rho_0}$} & \textbf{$\frac{a_3}{\rho_0}$} & \textbf{$\rho_0$}  \\ \hline
			4,03729 & 0,0619 & 3,125 & 1,52 & 0,515 & 0,7337 & 0,2663 & 0 & 0,513 \\ \hline
			4,16 & 0,06 & 3,181 & 1,581 & 0,5532 & 0,723 & 0,2662 & 0,011 & 0,52 \\ \hline
			4,31 & 0,058 & 3,2453 & 1,6521 & 0,6001 & 0,7119 & 0,2663 & 0,0218 & 0,5225 \\ \hline
			4,386 & 0,057 & 3,2791 & 1,689 & 0,6254 & 0,7063 & 0,2664 & 0,0273 & 0,5232 \\ \hline
			5 & 0,05 & 3,5396 & 1,98 & 0,8439 & 0,6662 & 0,2675 & 0,0663 & 0,5255 \\ \hline
			5,2632 & 0,0475 & 3,6641 & 2,1014 & 0,9424 & 0,6516 & 0,268 & 0,0804 & 0,5248 \\ \hline
			6,25 & 0.04 & 4,0893 & 2,5467 & 1,3268 & 0,6073 & 0,2696 & 0,1232 & 0,5152 \\ \hline
			6,3176 & 0,0396 & 4,1183 & 2,5728 & 1,3537 & 0,6047 & 0,2697 & 0,1256 & 0,5142 \\ \hline
		\end{tabular}
		\caption{The numerical solutions of system~\eqref{eq:zi-tri}. The value $\tau=0,5$ was kept constant. The values $z_i, a_i$ changed depending on $\mu$. However, one can appreciate that $\rho_0$ suffers very little change and moves around the value $1-\tau$}
		\label{table-tri}
	\end{table}
	
	In the last case present in the table a fourth zero point appeared at
  the point $z_4\sim 0,7123$, suggesting that this is the furthest we can go with trimorphism. For $\mu>\mu_3$, the fitness function $F(z)$ presented positive values, while for all $\mu_2<\mu\leq \mu_3$ that where tested proper trimorphism was obtained. Here we show some pictures of $F(z)$ both in proper trimorphism and when tetramorphism is attained. In order to better appreciate the details of the function, a signed power $F^{0,2}$ is also displayed. 
	
	\begin{figure}[H]
		\centering
        \renewcommand\thesubfigure{\Alph{subfigure}}
		\subfloat[Trimorphism with $\mu=6,25$]{ \includegraphics[scale = 0.55]{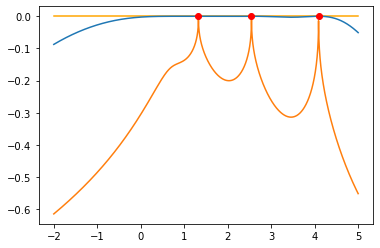}}
		\subfloat[Tetramorphism with $\mu=6,3176$]{ \includegraphics[scale = 0.55]{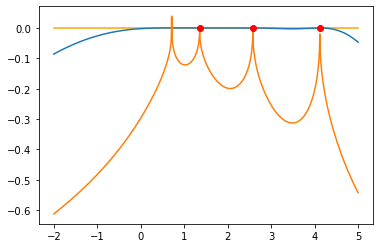}}
		\caption{Fitness functions $F$ (blue) and $F^{0,2}$ (orange, signed power) with $z_1, z_2,z_3$ highlighted in red. The parameters where taken from the last two rows of Table 1. As we see, for $\mu=6,25$ we obtain trimorphism, while for $\mu=6,3176$ the trimorphic fitness function attains positive values, meaning that the ESS is not trimorphic.}
	\end{figure}
	
	 		\appendix
	\section{The numerical scheme for the resolution of \eqref{eq:time_dependant_main}}\label{section:annex}
	


	Our code is an adaptation of the finite different schemes used in~\cite{CalvezHivertYoldacs, CalvezFigueroaAl}. In~\cite{CalvezFigueroaAl} a slightly different model, where the mutations were modeled by an integral kernel instead of a diffusion term, was studied numerically. The  scheme that is introduced in this article is expected to be asymptotic preserving, that is with a stability condition  independent of the small parameter $\e$. No theoretical analysis of the scheme is however provided. A theoretical analysis of a closely related scheme was provided for a selection-mutation model, without any transfer term, in \cite{CalvezHivertYoldacs} considering a diffusion term for the mutations. Our numerical scheme follows the schemes provided in \cite{CalvezHivertYoldacs, CalvezFigueroaAl}. 
 
	
	

Note that a Hopf-cole transformation of the phenotypic density, that is $u_\e(t,z)=\e\log (n_\e(t,z))$, with $n_\e(t,z)$ the solution to \eqref{eq:time}, leads to the following equation
$$
 \p_t u_\e(t,z)	-\varepsilon \p^2_{zz}u_\varepsilon (t,z)= |\p_z u_\e|^2+R(z) -  \rho_\varepsilon(t)    +  \tau   \displaystyle\int_\R \frac{ n_\varepsilon( t,y)}{ \rho_\varepsilon}H( z - y)\ {\rm dy}.
$$
Our numerical scheme  is based on the resolution of this equation. We provide below the main ingredients of our finite difference scheme.

	 We consider uniform time and trait steps  $\Delta t$ and $\Delta z$.   The  range for the traits is chosen to be $[Z_{min}, Z_{max}]$ so that the number of  discretization points of the interval is given by
	$$
	N_z=\frac{Z_{max}-Z_{min}}{\Delta z}+1.
	$$
	while  the grid is defined by
	$$
	z_j=Z_{min}+j\Delta z,\quad \text{for }0\leq j\leq N_z-1.
	$$
	
	Similarly, we consider the time interval $[0, T_{max}]$, discretized with step size $\Delta t$, meaning that we obtain
	$$
	N_t=\frac{T_{max}}{\Delta t}+1
	$$
	time steps, or iterations of the main loop of the code, and the time grid is given by
	$$
	t_i = i\Delta t, \quad\text{for } 0 \leq i \leq N_t -1.
	$$

	
	\noindent\textit{Step 1) } We begin now the $N_t-1$ iterations of the main loop of the code where we compute each $u_j^i$. The numerical formula for each iteration reads
	\begin{equation}\label{eq:scheme}
		\frac{u^{i+1}_j-u^i_j}{\Delta t}=R(x_j)-\rho^{i+1} + T_j^i + \varepsilon \Delta u^i_j + (D u^i_j)^2.
	\end{equation}
	Let us explain each of the terms separately.
	
	The term $R(z_j)$ stands for the growth-death term. Considering the quadratic growth rate \eqref{R:quadratic}, we obtain
	$$
	R(z_j)=1-gz_j^2.
	$$
	
	The function $T_j^i$ stands for the transfer term, which we computed directly as a convolution between the transfer kernel and $n/\rho$. We define
	$$
	H(z_j):=\tanh(z_j)\quad \text{where }z_j\in [Z_{min}-Z_{max}, Z_{max}-Z_{min}]
	$$
	and
	$$
	\frac{n^i}{\rho^i}(z_j):=\frac{e^{u^i_j/\varepsilon}}{\Delta z\ \sum\limits_{j=0}^{N_z-1}e^{u^i_j/\varepsilon}}.
	$$
	The term $\Delta u_j^i$ stands for the Laplacian term and has been computed simply by
	$$
	\Delta u_j^i =\frac{u_{j-1}^i - 2u_j^i + u_{j+1}^i}{(\Delta z)^2}.
	$$
	It is worth noting at this point that in order to compute derivatives on the boundary of the interval, the function $u_j^i$ is extrapolated, at each iteration, one point to the left and one point to the right of the interval in the points $u^i_{-1}$ and $u^i_{N_x}$ defined by
	$$
	u^i_{-1}=4u^i_0 -6u^i_1 + 4u^i_2 - u^i_3 \quad\text{and}\quad u^i_{N_x}=4u^i_{N_x-1} -6u^i_{N_x-2} + 4u^i_{N_x-3} - u^i_{N_x-4}.
	$$
	
	Next, we determine the term $(D u^i_j)^2$. For this term, a classical monotone scheme has been considered which is suitable to capture the viscosity solution of the corresponding Hamilton-Jacobi equation at the limit as $\e\to 0$, see \cite{C.L:84}. The formula goes by
	$$
	(D u^i_j)^2:=\max\{(D u^+_j)^2, (D u^-_j)^2\}
	$$
	where
	$$
	(D u^+_j)^2:=\begin{cases}
		0 &\text{if } \nabla^+_j u \geq 0\\
		(\nabla^+_j u)^2 \quad &\text{if }\nabla^+_j u < 0
	\end{cases},\quad \text{with }\nabla^+_j u:=\frac{u_j^i-u_{j-1}^i}{\Delta z}
	$$
	and
	$$
	(D u^-_j)^2:=\begin{cases}
		0 &\text{if } \nabla^-_j u \leq 0\\
		(\nabla^-_j u)^2 \quad &\text{if }\nabla^-_j u > 0
	\end{cases},\quad \text{with }\nabla^-_j u:=\frac{u_{j+1}^i-u_{j}^i}{\Delta z}.
	$$
	
	At this point we are ready to define an intermediate function
	$$
	A^i_j:= u^i_j + \Delta t \left( R(z_j) + T_j^i + \varepsilon \Delta u^i_j + (D u^i_j)^2 \right)
	$$
	since the computations of $\rho^{i+1}$ is a bit more complicated.
	
	\noindent\textit{Step 2) }Let us compute now $\rho^{i+1}$. From equation~\eqref{eq:scheme} we deduce
	$$
	e^{u_j^{i+1}/\varepsilon}= e^{- \rho^{i+1}\Delta t/\varepsilon} e^{A^i_j/\varepsilon}.
	$$
	Integrating here we obtain the following closed equation
	\begin{equation}\label{eq:mass_scheme}
		\rho^{i+1}=e^{- \rho^{i+1}\Delta t/\varepsilon}\  \Delta x\ \sum\limits_{j=0}^{N_z-1}e^{A_j^i/\varepsilon}.
	\end{equation}
	From here we deduce that $\rho^{i+1}$ must be the root of the function
	$$
	h(y):=ye^{Cy\Delta t/\varepsilon}-\Delta z\  \sum\limits_{j=0}^{N_z-1}e^{A_j^i/\varepsilon}.
	$$
	Equivalently, one can take logarithms in equation~\eqref{eq:mass_scheme} and arrive to
	$$
	g(y):=-\varepsilon \ln y - C\Delta t y +\varepsilon\ln \left(\Delta z\  \sum\limits_{j=0}^{N_z-1}e^{A_j^i/\varepsilon}\right).
	$$
    The total population size $\rho^{i+1}$ is the root of both functions $h$ and $g$. It is worth noting that the method used for finding such roots is Newton's, and since in both cases the functions are smooth convex functions of $y$ the method converges.
	
	The reason why both approaches (with and without logarithms) are necessary is the following. The function $h$ is to be chosen when $\rho^{i+1}$ is close to 0 (for large values it becomes less accurate), whereas $g$ is more adapted when $\rho^{i+1}$ is not small, since it is more prone to accumulate numerical errors when $\rho^{i+1}\to 0$. In the effective implementation of the method, either one formulation or the other is chosen, depending on the values reached during the iterations of the algorithm.
	
	\noindent\textit{Step 3) } We are ready now to compute
	$$
	u_j^{i+1}=A_j^{i} - \Delta t \rho^{i+1}.
	$$

\bigskip
\noindent \textbf{Acknowledgements}:
All authors would like to thank Sylvie M\'el\'eard for the introduction of the problem, fruitful discussions and early computations. The authors also thank H\'el\`ene Hivert for fruitfull discussions concerning the numerical simulations.
This work has been supported by  the ANR project DEEV ANR-20-CE40-0011-01, the Chair ``Modélisation Mathématique et Biodiversité" of Veolia Environnement-Ecole Polytechnique-Museum National d’Histoire Naturelle-Fondation X. This project has also been partially funded by the European Union (ERC-AdG SINGER-101054787). Views and opinions expressed are however those of the
authors only and do not necessarily reflect those of the European Union or the European
Research Council. Neither the European Union nor the granting authority can be held responsible
for them.

\end{document}